\RequirePackage{fix-cm}
\documentclass[smallcondensed]{svjour3}     % onecolumn (ditto)
\smartqed  % flush right qed marks, e.g. at end of proof
\usepackage{graphicx}
\usepackage{epstopdf}
\usepackage{amssymb}
\usepackage{amsmath}
\usepackage{enumerate}
 \usepackage{mathptmx}      % use Times fonts if available on your TeX system
\usepackage[none]{hyphenat}
\spnewtheorem{algorithm}{Algorithm}{\bf}{\rm}
 \journalname{Journal of Global Optimization}
\begin{document}

\title{Extended reverse-convex programming: an approximate enumeration approach to global optimization%\thanks{Grants or other notes
}

\titlerunning{Extended reverse-convex programming}        % if too long for running head

\author{Gene A. Bunin %etc.
}


\institute{
              \email{gene.a.bunin@ccapprox.info}   \\
     \emph{Current affiliation: Xinjiang Arts Institute, 734 Tuanjie Road, Urumqi, Xinjiang Uyghur Autonomous Region, People's Republic of China, 830000}  %  if needed
}

\date{Received: date / Accepted: date}

\maketitle

\begin{abstract}
A new approach to solving a large class of factorable nonlinear programming (NLP) problems to global optimality is presented in this paper. Unlike the traditional strategy of partitioning the decision-variable space employed in many branch-and-bound methods, the proposed approach approximates the NLP problem by a reverse-convex programming (RCP) problem to a controlled precision, with the latter then solved by an enumerative search. To establish the theoretical guarantees of the method, the notion of ``RCP regularity'' is introduced and it is proven that enumeration is guaranteed to yield a global optimum when the RCP problem is regular. An extended RCP algorithmic framework is then presented and its performance is examined for a small set of test problems. 
\end{abstract}

\section{Introduction}
\label{intro}

Consider the following nonlinear programming (NLP) problem:

\vspace{-2mm}
\begin{equation}\label{eq:mainprob}
\begin{array}{rl}
\mathop {\rm minimize}\limits_{y} & f_{NL} (y) \\
{\rm{subject}}\;{\rm{to}} & g_{NL,i}(y) \leq 0,\hspace{3mm}i = 1,...,n_{g_{NL}} \\
& h_{NL,i}(y) = 0,\hspace{3mm}i = 1,...,n_{h_{NL}},
\end{array}
\end{equation}

\noindent with $y \in \mathbb{R}^{n_y}$ denoting the decision variables and $f_{NL}, g_{NL,i}, h_{NL,i} : \mathbb{R}^{n_y} \rightarrow \mathbb{R}$ denoting the cost, inequality constraint, and equality constraint functions, respectively. In the present work, we will restrict our attention to \emph{factorable NLP problems} -- i.e., to problems where the functions $f_{NL}$, $g_{NL,i}$, and $h_{NL,i}$ are factorable.

\begin{definition}[Factorable functions]\label{def:factor} Following the definition of Sherali and Wang \cite{Sherali:01}, the function $f : \mathbb{R}^{n_y} \rightarrow \mathbb{R}$ is called \emph{factorable} if it can be written as a sum of $m$ products of univariate functions $\phi_{ij}$:

\vspace{-2mm}
$$
f(y) = \sum_{i=1}^{m} \prod_{j=1}^{n_y} \phi_{ij}(y_j).
$$

\end{definition}

Branch-and-bound methods that intelligently partition the decision-variable space have, over the past half-century, emerged as the dominant technique for solving the general factorable NLP problem (\ref{eq:mainprob}) \cite{Neumaier:04,Neumaier:05}. This is, in part, due to the natural applicability of the approach, as convex relaxations are easily obtained for the univariate components $\phi$ \cite{Falk:69,McCormick:76}, with strong relaxations for certain multivariate functions also available \cite{Zamora:99,Floudas:05}. Additionally, a number of algorithmic advances -- notably, the domain reduction techniques used by solvers such as BARON \cite{Ryoo:96,Tawarmalani:04,Caprara:10} -- and general improvements in computational power have made the resulting solvers viable for an increasingly greater number of practical problems. However, even with these successes such methods can fall to the \emph{curse of dimensionality} due to their need to continuously partition the decision-variable space, which becomes difficult to do efficiently with an increasing number of variables.

While there is no apparent way to ``break'' the curse, it is important to be aware that the fundamental nature of the curse may differ depending on the nature of the algorithm and the optimization problem. Notable examples include mixed integer \cite{Nowak:05} and concave minimization \cite{McKeown:76,Pardalos:86,Horst:95} problems, for which the potential solutions may be enumerated and then compared to obtain a global optimum. Although obtaining a single solution candidate is often computationally cheap, the difficulty in such methods is generally due to the number of candidates becoming unacceptably large. A common motif in such problems is that an optimum must lie at the intersection of $n_y$ constraints, which may be integral constraints or linear inequalities. Consequently, the number of candidates to check reaches -- if one lets $n_c$ denote the total number of constraints and supposes the worst case -- the binomial coefficient $\binom{n_c}{n_y}$. Clearly, the curse of dimensionality in this case is due not only to $n_y$ but rather to both $n_y$ and $n_c$. Of interest is the observation that for problems where the gap between $n_y$ and $n_c$ is innately bounded, or can be made bounded, the worst-case computational effort is bounded as $O(n_y^{n_c-n_y})$: e.g., $\binom{n_c}{n_y} = n_y+1$ when $n_c = n_y+1$ and $\binom{n_c}{n_y} = 0.5(n_y+2)(n_y+1)$ when $n_c = n_y+2$. The natural conclusion is that enumeration techniques may scale and perform better for certain problems than methods relying on the partitioning of the decision-variable space.

The main contribution of this work consists in proposing a framework where the factorable NLP problem is approximated by a problem for which enumeration may be applied. Namely, one approximates (\ref{eq:mainprob}) by the reverse-convex programming (RCP) problem

\vspace{-2mm}
\begin{equation}\label{eq:RCPprob}
\begin{array}{rl}
\mathop {\rm minimize}\limits_{x} & c^T x \\
{\rm{subject}}\;{\rm{to}} & g_i(x) \leq 0,\hspace{3mm}i = 1,...,n_g \\
 & Cx = d,
\end{array}
\end{equation} 

\noindent where $x \in \mathbb{R}^{n}$ denotes the vector of variables, $c \in \mathbb{R}^{n}$ the cost vector, $g_i : \mathbb{R}^{n} \rightarrow \mathbb{R}$ a set of $n_g$ \emph{concave} (``reverse-convex'') constraint functions, and $C \in \mathbb{R}^{n_C \times n}$, $d \in \mathbb{R}^{n_C}$ the matrix and vector defining the linear equality constraints. The formulation in (\ref{eq:RCPprob}) will be referred to as the ``standard'' RCP form. The choice to use a linear cost function is merely a preference that allows one to lump all of the nonlinearity into the inequality constraints, equivalent formulations with a general concave cost also being possible \cite{Ueing:72,Hillestad:80}.

The key message that this paper aims to convey is thus the following:

\begin{quote}
\emph{By solving the RCP approximation (\ref{eq:RCPprob}), one may solve the general factorable NLP problem (\ref{eq:mainprob}) by an enumerative method for which the curse of dimensionality is \emph{different} than for schemes that branch directly on the decision-variable space.}
\end{quote}

\noindent So as to avoid misleading the reader, it must be noted that ``different'' does not imply ``better''. However, the proposed approach may be seen as a potential alternative for those problems where branching on the decision-variable space proves inefficient. Because the RCP formulation is used as a tool for solving problems for which it was not originally intended, the term \emph{extended reverse-convex programming} is used.

In presenting the extended RCP framework, the quality of the approximation

\vspace{-2mm}
$$
\begin{array}{rl}
\mathop {\rm minimize}\limits_{y} & f_{NL} (y) \\
{\rm{subject}}\;{\rm{to}} & g_{NL,i}(y) \leq 0,\hspace{3mm}i = 1,...,n_{g_{NL}} \\
& h_{NL,i}(y) = 0,\hspace{3mm}i = 1,...,n_{h_{NL}}
\end{array} \;\;\; \approx \;\;\;
\begin{array}{rl}
\mathop {\rm minimize}\limits_{x} & c^T x \\
{\rm{subject}}\;{\rm{to}} & g_i(x) \leq 0,\hspace{3mm}i = 1,...,n_g \\
 & Cx = d
\end{array}
$$

\noindent is addressed first in Section 2, and it is proven that this approximation may be arbitrarily good provided that:

\begin{enumerate}
{\setlength\itemindent{.65cm} \item[A1.] The functions $\phi$ obtained in the decomposition of $f_{NL}$, $g_{NL,i}$, and $h_{NL,i}$ are Lipschitz-continuous over any finite interval.}
{\setlength\itemindent{.65cm} \item[A2.] The feasible domain of (\ref{eq:mainprob}), 

\vspace{-2mm}
$$\mathcal{Y} = \{ y : g_{NL,i} (y) \leq 0, \; i = 1,...,n_{g_{NL}}; \; h_{NL,i} (y) = 0, \; i = 1,...,n_{h_{NL}} \},$$ 

\noindent is bounded.}
\end{enumerate}

\noindent An algorithm to obtain an arbitrarily good approximation is also provided. 

The properties of global solutions of (\ref{eq:RCPprob}) are then discussed in Section 3. Letting $\mathcal{X} = \{ x : g_i (x) \leq 0, \; i = 1,...,n_g; \; Cx = d \}$ denote the feasible domain of (\ref{eq:RCPprob}), the following assumptions are imposed for the case when $\mathcal{X} \neq \varnothing$:

\begin{enumerate}
{\setlength\itemindent{.65cm} \item[B1.] Each $g_i$ is concave and continuously differentiable over an open set containing $\mathcal{X}$.}
{\setlength\itemindent{.65cm} \item[B2.] The rank of $C$ is $n_C$, with $n_C < n$.}
{\setlength\itemindent{.65cm} \item[B3.] $c^T x$ attains its global minimum at $x^* \in \mathcal{X}$.}
{\setlength\itemindent{.65cm} \item[B4.] The linear independence constraint qualification (LICQ) holds at $x^*$.}
\end{enumerate}

\noindent Generalizing previously reported results \cite{Hillestad:80}, the notion of ``RCP regularity'' is introduced so that an optimal solution $x^*$ to any regular RCP problem can be obtained by solving the convex problem

\vspace{-2mm}
\begin{equation}\label{eq:RCPprobrev}
\begin{array}{rl}
x^* = {\rm arg} \; \mathop {\rm maximize}\limits_{x} & c^T x \\
{\rm{subject}}\;{\rm{to}} & g_i(x) \geq 0,\hspace{3mm}\forall i \in i_{A^*} \\
 & Cx = d,
\end{array}
\end{equation} 

\noindent where $i_{A^*}$ is the index set of $n-n_C$ constraints that are active at $x^*$. From this crucial property follows the familiar combinatorial motif, as one may find $x^*$ by enumerating all of the $\binom{n_g}{n-n_C}$ possible active sets and solving the problem 

\begin{equation}\label{eq:RCPprobrevgen}
\begin{array}{rl}
\mathop {\rm maximize}\limits_{x} & c^T x \\
{\rm{subject}}\;{\rm{to}} & g_i(x) \geq 0,\hspace{3mm}\forall i \in i_{A} \\
 & Cx = d
\end{array}
\end{equation} 

\noindent for each candidate $i_A$.

Finally, Section 4 proposes an extended RCP algorithm that performs better than the brute enumeration of the $\binom{n_g}{n-n_C}$ possibilities. At its core, this algorithm is similar to that proposed by Ueing \cite{Ueing:72} in that it builds up to the set of $n-n_C$ active constraints by first constructing their subsets in a familiar branching-and-fathoming manner. An algorithmic contribution of this work consists in adding a number of elements to speed up this scheme. This comes in the form of several fathoming techniques that quickly eliminate active-set candidates that cannot occur at $x^*$. Some of these techniques are original in the sense that they make use of the extended RCP framework explicitly in fathoming certain sets. Others, such as the powerful technique of domain reduction \cite{Tawarmalani:04,Caprara:10}, are already well established but are nevertheless original in how they are applied here. The strengths and drawbacks of the algorithm are then demonstrated for a small set of test problems in Section~5, with general reflections and an outline of future work concluding the paper in Section~6.

Some notes with regard to notation and terminology:

\begin{itemize}
\item All vectors are, unless otherwise stated, column vectors.
\item Given a vector $x$, $x_i$ will refer to its $i^{\rm th}$ element. Given a matrix $X$, $X_{i.}$ will refer to its $i^{\rm th}$ row. If the vector or matrix already has a subscript (e.g., $x_a$ or $X_a$), a comma will be used to separate the index -- i.e., $x_{a,i}$ denoting the $i^{\rm th}$ element of $x_a$ and $X_{a,i.}$ the $i^{\rm th}$ row of $X_a$. The notation $X_{ij}$ will refer to the element of $X$ corresponding to the $i^{\rm th}$ row and the $j^{\rm th}$ column, with parentheses used to avoid ambiguity when needed -- e.g., $X_{i(j+1)}$ denoting the element of $X$ corresponding to the $i^{\rm th}$ row and the $({j+1})^{\rm th}$ column.
\item The matrix $I_{n}$ will denote an $n \times n$ identity matrix. The dagger symbol, $\dagger$, will denote the Moore-Penrose pseudoinverse.
\item The symbol $\subseteq_B$ will be used to indicate that a given binary vector is a member of another binary vector in the sense that all of the elements that are equal to unity in the former are also so in the latter, e.g., $[0\;1\;1\;0\;0] \subseteq_B [0\;1\;1\;1\;0]$, but $[0\;1\;1\;0\;0] \not \subseteq_B [0\;1\;0\;1\;0]$.
\item The symbol \# will be employed to denote the cardinality of a set.
\item The matricial ``max norm'', $\| \cdot \|_{\rm max}$, will be used to denote the maximum of the absolute values of the matrix elements, i.e., $\displaystyle \| X \|_{\rm max} = \mathop {\max}_{i,j} | X_{ij} |$.
\item The adjectives ``reverse-convex'' and ``concave'' are identical and will be used interchangeably throughout the text.
\end{itemize}

\section{The RCP Approximation}
\label{sec:pwapp}

The idea of approximating the general NLP (\ref{eq:mainprob}) by the RCP problem (\ref{eq:RCPprob}) has its roots in the works of Zangwill \cite{Zangwill:67} and Rozvany \cite{Rozvany:67,Rozvany:70,Rozvany:71}. Both authors mention the possibility of approximating a more general function by a \emph{piecewise-concave function}, which may in turn be easily decomposed into a set of reverse-convex (concave) ones. The procedure outlined here is essentially a two-step process consisting of (i) obtaining a factored decomposition of (\ref{eq:mainprob}) and (ii) approximating the nonlinear components of the decomposition by a set of concave functions.

\subsection{The Factored Decomposition of the NLP Problem (\ref{eq:mainprob})} 

The decomposition outlined here largely follows that already discussed in the literature \cite{Tawarmalani:04,Nowak:05}. Namely, it is shown that (\ref{eq:mainprob}) may be decomposed into an equivalent problem with a linear cost, linear equalities, and a finite number of inequality constraints whose nonlinear elements are either univariate or bilinear. The key result presented next establishes equivalence for a single factorable constraint.

\begin{lemma}[Decomposing a factorable constraint]\label{lemma:decomp} Consider the constraint $f(y) \leq 0$, where the function $f$ is factorable as per Definition \ref{def:factor}. This constraint may be replaced by an equivalent set of $m (4n_y-3)$ inequalities and $m$ linear equalities. The nonlinear terms will be univariate in $2m n_y$ of the inequality constraints and will be bilinear in the rest. 
\end{lemma}
\begin{proof} Employing Definition \ref{def:factor}, the constraint is first rewritten as

\vspace{-2mm}
\begin{equation}\label{eq:fproof1}
\sum_{i=1}^{m} \prod_{j=1}^{n_y} \phi_{ij}(y_j) \leq 0.
\end{equation}

\noindent Introducing $m-1$ auxiliary variables, denoted by $z_{a,1},...,z_{a,m-1}$, and employing the epigraph transformation yields the constraint set

\vspace{-2mm}
\begin{equation}\label{eq:fproof2}
\begin{array}{l}
\displaystyle \prod_{j=1}^{n_y} \phi_{1j}(y_j) + \sum_{i=1}^{m-1} z_{a,i} \leq 0 \vspace{1mm} \\
\displaystyle \prod_{j=1}^{n_y} \phi_{ij}(y_j) - z_{a,i-1} \leq 0, \;\; i = 2,...,m,
\end{array}
\end{equation}

\noindent which is equivalent to (\ref{eq:fproof1}) as any $y, z_a$ satisfying (\ref{eq:fproof2}) implies that (\ref{eq:fproof1}) is satisfied, and as for any $y$ satisfying (\ref{eq:fproof1}) there always exists a choice of $z_a$ (namely, $z_{a_,i-1} = \prod_{j=1}^{n_y} \phi_{ij}(y_j), \; i = 2,...,m$) that leads to the satisfaction of (\ref{eq:fproof2}).

Consider the matrix of auxiliary variables $Z_b \in \mathbb{R}^{m \times n_y}$, with the restriction that $\phi_{ij}(y_j) - Z_{b,ij} = 0$. The set (\ref{eq:fproof2}) may then be replaced by the equivalent set

\vspace{-2mm}
\begin{equation}\label{eq:fproof3}
\begin{array}{l}
\displaystyle \prod_{j=1}^{n_y} Z_{b,1j} + \sum_{i=1}^{m-1} z_{a,i} \leq 0 \vspace{1mm} \\
\displaystyle \prod_{j=1}^{n_y} Z_{b,ij} - z_{a,i-1} \leq 0, \;\; i = 2,...,m \vspace{1mm} \\
\displaystyle \phi_{ij}(y_j) - Z_{b,ij} = 0, \;\; i = 1,...,m;\; j = 1,...,n_y.
\end{array}
\end{equation}

Finally, following the introduction of the auxiliary-variable matrix $Z_c \in \mathbb{R}^{m \times (n_y-1)}$, the product term $\prod_{j=1}^{n_y} Z_{b,ij}$ may be decomposed by recursive substitution as follows:

\vspace{-2mm}
$$
\begin{array}{l}
\displaystyle \prod_{j=1}^{n_y} Z_{b,ij} = Z_{b,i1} Z_{c,i1} \vspace{1mm} \\
Z_{c,ij} - Z_{b,i(j+1)} Z_{c,i(j+1)} = 0, \;\; j = 1,...,n_y-2 \vspace{1mm} \\
Z_{c,i(n_y-1)} - Z_{b,in_y} = 0,
\end{array}
$$

\noindent thereby allowing for (\ref{eq:fproof3}) to be replaced by the equivalent set

\vspace{-2mm}
$$
\begin{array}{l}
\displaystyle Z_{b,11} Z_{c,11} + \sum_{i=1}^{m-1} z_{a,i} \leq 0 \vspace{1mm} \\
\displaystyle Z_{b,i1} Z_{c,i1} - z_{a,i-1} \leq 0, \;\; i = 2,...,m \vspace{1mm} \\
\displaystyle Z_{c,ij} - Z_{b,i(j+1)} Z_{c,i(j+1)} = 0, \;\; i = 1,...,m;\; j = 1,...,n_y-2   \vspace{1mm} \\
Z_{c,i(n_y-1)} -  Z_{b,in_y} = 0, \;\; i = 1,...,m \vspace{1mm} \\
\displaystyle \phi_{ij}(y_j) - Z_{b,ij} = 0, \;\; i = 1,...,m;\; j = 1,...,n_y.
\end{array}
$$

Breaking the nonlinear equalities then yields

\vspace{-2mm}
\begin{equation}\label{eq:decompcon}
\begin{array}{l}
\displaystyle Z_{b,11} Z_{c,11} + \sum_{i=1}^{m-1} z_{a,i} \leq 0 \vspace{1mm} \\
\displaystyle Z_{b,i1} Z_{c,i1} - z_{a,i-1} \leq 0, \;\; i = 2,...,m \vspace{1mm} \\
\displaystyle Z_{c,ij} - Z_{b,i(j+1)} Z_{c,i(j+1)} \leq 0, \;\; i = 1,...,m;\; j = 1,...,n_y-2   \vspace{1mm} \\
\displaystyle Z_{b,i(j+1)} Z_{c,i(j+1)} - Z_{c,ij} \leq 0, \;\; i = 1,...,m;\; j = 1,...,n_y-2   \vspace{1mm} \\
Z_{c,i(n_y-1)} -  Z_{b,in_y} = 0, \;\; i = 1,...,m \vspace{1mm} \\
\displaystyle \phi_{ij}(y_j) - Z_{b,ij} \leq 0, \;\; i = 1,...,m;\; j = 1,...,n_y \vspace{1mm} \\
\displaystyle -\phi_{ij}(y_j) + Z_{b,ij} \leq 0, \;\; i = 1,...,m;\; j = 1,...,n_y,
\end{array}
\end{equation}

\noindent with the final $2mn_y$ constraints nonlinear only in the univariate functions $\phi$, and the other inequality constraints nonlinear only in the bilinear terms. \qed

\end{proof}

The decomposition of the entire NLP problem (\ref{eq:mainprob}) follows readily.

\begin{corollary}[Factored decomposition of the NLP problem (\ref{eq:mainprob})]\label{cor:NLPfactor} The factorable NLP problem (\ref{eq:mainprob}) may be decomposed into an equivalent problem with a linear cost and a feasible set defined by a finite number of linear equality and nonlinear inequality constraints. The nonlinear terms in the inequalities are either univariate or bilinear.
\end{corollary}
\begin{proof}
By the epigraph transformation \cite{McCormick:76} and the splitting of equality constraints, one first obtains the equivalence

\vspace{-2mm}
\begin{equation}\label{eq:mainprobeq}
\begin{array}{l}
\begin{array}{rl}
\mathop {\rm minimize}\limits_{y} & f_{NL} (y) \\
{\rm{subject}}\;{\rm{to}} & g_{NL,i}(y) \leq 0,\hspace{3mm}i = 1,...,n_{g_{NL}} \\
& h_{NL,i}(y) = 0,\hspace{3mm}i = 1,...,n_{h_{NL}}
\end{array} \vspace{2mm} \\
\hspace{25mm} \Leftrightarrow \;\;\;
\begin{array}{rl}
\mathop {\rm minimize}\limits_{y,t} & t \\
{\rm{subject}}\;{\rm{to}} & f_{NL} (y) - t \leq 0 \\
& g_{NL,i}(y) \leq 0,\hspace{3mm}i = 1,...,n_{g_{NL}} \\
& h_{NL,i}(y) \leq 0,\hspace{3mm}i = 1,...,n_{h_{NL}} \\
& -h_{NL,i}(y) \leq 0,\hspace{3mm}i = 1,...,n_{h_{NL}},
\end{array}
\end{array}
\end{equation}

\noindent Since the functions $f_{NL} (y)$, $g_{NL,i}(y)$, $h_{NL,i}(y)$, and $-h_{NL,i}(y)$ are all factorable, the result of Lemma \ref{lemma:decomp} may be exploited for each. Note that the addition of the epigraph variable, $t$, does not affect the derivation or the validity of Lemma 1 -- i.e., one simply replaces $Z_{b,11} Z_{c,11} + \sum_{i=1}^{m-1} z_{a,i} \leq 0$ by $Z_{b,11} Z_{c,11} + \sum_{i=1}^{m-1} z_{a,i} - t \leq 0$ in (\ref{eq:decompcon}). \qed
\end{proof}

\subsection{Approximation by the Piecewise-Concave Function}

Having decomposed the original NLP problem into an equivalent problem with a linear cost function and constraint functions whose nonlinearity only appears via univariate and bilinear terms, it is now possible to approximate this problem by a reverse-convex one by obtaining reverse-convex (concave) approximations of the univariate and bilinear elements. To make the link between the two problems more explicit, let $x = (y,t,z_{all})$, where $z_{all}$ is the vector of all the auxiliary variables ($z_a$, $Z_b$, $Z_c$) added during the decomposition of \emph{all} of the NLP problem elements.

The key tool for carrying out the approximation is the piecewise-concave function, defined by Zangwill \cite{Zangwill:67} as

\vspace{-2mm}
$$
p(x) = \mathop {\max} \limits_{i = 1,...,n_p} p_i (x),
$$

\noindent with the restriction that all $p_i : \mathbb{R}^{n} \rightarrow \mathbb{R}$ are concave.

The ability of the piecewise-concave function to act as an arbitrarily good approximation of several commonly encountered functions was discussed and proven in a supplement to the present work \cite{Bunin:2014}. The following two results will be needed here.

\begin{lemma}[Piecewise-concave approximation of a Lipschitz-continuous univariate function over a bounded interval]\label{lemma:pwappuni} For a given univariate function $\phi$ of variable $x_i$ and for $\epsilon_p > 0$, there exists a piecewise-concave approximation $p : \mathbb{R} \rightarrow \mathbb{R}$ such that

\vspace{-2mm}
\begin{equation}\label{eq:pwapp}
\mathop {\max}_{x_i \in [\underline x_i, \overline x_i]} | p(x_i) - \phi (x_i) | \leq \epsilon_p,
\end{equation}

\noindent provided that $\phi$ is Lipschitz-continuous on the interval $[\underline x_i, \overline x_i]$:

\vspace{-2mm}
$$
| \phi (x_{i}^a) - \phi (x_{i}^b) | < \kappa | x_{i}^a - x_{i}^b |, \;\; \forall x_{i}^a, x_{i}^b \in [\underline x_i, \overline x_i] \;\; (x_{i}^a \neq x_{i}^b),
$$

\noindent with $\kappa < \infty$ denoting the Lipschitz constant of $\phi$.
\end{lemma}
\begin{proof}
The proof proceeds by showing that one can always approximate $\phi$ by a sequence of evenly spaced concave parabolas, with the approximation becoming arbitrarily good as the number of parabolas increases. The reader is referred to \cite{Bunin:2014}. \qed
\end{proof}

\begin{lemma}[Piecewise-concave approximation of a bilinear function over a box]\label{lemma:pwappbilinear} For the bilinear function $\pm x_i x_j$  and for $\epsilon_p > 0$, there exists a piecewise-concave approximation $p : \mathbb{R}^2 \rightarrow \mathbb{R}$ such that

\vspace{-2mm}
\begin{equation}\label{eq:pwappbl}
\mathop {\max}_{\footnotesize{ \begin{array}{c} x_i \in [\underline x_i, \overline x_i] \\ x_j \in [\underline x_j, \overline x_j] \end{array} }} | p(x_i, x_j) \mp x_i x_j | \leq \epsilon_p.
\end{equation}

\end{lemma}
\begin{proof} The bilinear term may be decomposed into the difference-of-convex (D. C.) form as $\pm x_i x_j = \pm 0.5(x_i+x_j)^2 \mp 0.5x_i^2 \mp 0.5x_j^2$, for which (\ref{eq:pwappbl}) is a specific case of a more general result where the convex component is approximated by a piecewise-linear function \cite{Bunin:2014}. \qed   
\end{proof}

In approximating the NLP problem, one may prefer either an outer or an inner approximation. The former, obtained by \emph{underapproximating} all of the nonlinear components, is desirable as it guarantees that feasibility is not lost in the approximation, with feasibility of the NLP problem guaranteeing the feasibility of the RCP. On the other hand, \emph{overapproximating} is potentially useful as any solution that solving the RCP problem returns is guaranteed to be feasible for the original problem. The following lemma establishes that one can obtain either kind of approximation while maintaining precision.

\begin{lemma}[Under- and overapproximations]\label{lemma:underover} Let $P(x,\cdot)$ be a set of approximations satisfying

\vspace{-2mm}
$$
\mathop {\max}_{x \in \mathcal{X}} | P(x,\epsilon_p) - f (x) | \leq \epsilon_p
$$
 
\noindent for all $\epsilon_p > 0$. It follows that there exist under- and overapproximations -- denoted by $p^-$ and $p^+$, respectively -- such that, for $\epsilon_p^-, \epsilon_p^+ > 0$,

\vspace{-2mm}
$$
f(x) - \epsilon_p^{-} \leq p^- (x) \leq f(x) \leq p^+ (x) \leq f(x) + \epsilon_p^{+}, \; \forall x \in \mathcal{X}.
$$.

\end{lemma}
\begin{proof} Set $p^- (x) = P(x,0.5\epsilon_p^{-}) - 0.5\epsilon_p^{-}$ and $p^+ (x) = P(x,0.5\epsilon_p^{+}) + 0.5\epsilon_p^{+}$. \qed
\end{proof}

The main approximation theorem follows and states that a factorable NLP problem decomposed by the procedure outlined in Lemma \ref{lemma:decomp} and Corollary \ref{cor:NLPfactor} has an outer RCP approximation, and that the solution of this approximation yields a solution $y_{app}^*$ that (a) has a lower cost function value than $y^*$ and (b) satisfies the NLP problem constraints to tolerances that go to ${\bf 0}$ as $\epsilon_p^- \rightarrow 0$ (i.e., as more precise piecewise-concave approximations are used). The proof is long, but may be outlined as:

\begin{enumerate}[(i)]
\item The constraint family (\ref{eq:decompcon}) appears $1+n_{g_{NL}}+2n_{h_{NL}}$ times in the decomposition of (\ref{eq:mainprobeq}), and the RCP approximation is derived by showing how piecewise-concave underapproximations may be obtained for all of the nonlinear elements of (\ref{eq:decompcon}). This requires first verifying that the domains of the variables remain bounded, and then applying Lemmas \ref{lemma:pwappuni} and \ref{lemma:pwappbilinear}.
\item Having derived the RCP approximation, it is then proven that the solution that results, $y_{app}^*$, must have a lower cost function value than $y^*$ because of the approximation having a relaxed constraint set.
\item Finally, it is shown that the tolerances with which $y_{app}^*$ satisfies the constraints of the NLP problem are upper bounded by finite-degree polynomials of the individual approximation errors (i.e., the specific errors from the interval $[-\epsilon_p^-, \epsilon_p^-]$ incurred at the solution for the different elements of (\ref{eq:decompcon})). For simplicity, all such individual errors will be denoted by the single scalar $\zeta_p$, with the vectors $\zeta_{g_i}$, $\zeta_{h_i}^+$, and $\zeta_{h_i}^-$ denoting the collections of errors incurred for the decomposed sets of $g_{NL,i}$, $h_{NL,i}$, and $-h_{NL,i}$, respectively. Clearly, $\zeta_{g_i}, \zeta_{h_i}^+, \zeta_{h_i}^- \rightarrow {\bf 0}$ as $\epsilon_p^- \rightarrow 0$.
\end{enumerate}

\begin{theorem}[Reverse-convex approximation of the NLP problem (1)]\label{thm:mainapprox} Let Problem (\ref{eq:mainprob}) be decomposed into an  equivalent problem with a linear cost, linear equality constraints, and inequality constraints as outlined in Lemma \ref{lemma:decomp} and Corollary \ref{cor:NLPfactor}. It follows that

\begin{enumerate}[(i)]
\item there exists an outer RCP approximation of the decomposed problem with the global solution $x^*$, such that the approximate solution $y_{app,i}^* = x_i^*$ ($i = 1,...,n_y$) satisfies
\item $f_{NL} (y_{app}^*) \leq f_{NL} (y^*)$,
\item $g_{NL,i} (y_{app}^*) \leq q_{g,i}(\zeta_{g_i})$, $\forall i = 1,...,n_{g_{NL}}$, and $|h_{NL,i} (y_{app}^*)|$ $\leq \mathop {\max} \left( q_{h,i}^+ (\zeta_{h_i}^+), q_{h,i}^- (\zeta_{h_i}^-) \right),$ $ \forall i = 1,...,n_{h_{NL}}$, where $q_{g,i}$, $q_{h,i}^+$, and $q_{h,i}^-$ are finite-degree polynomials satisfying $q_{g,i}({\bf 0}) = q_{h,i}^+({\bf 0}) = q_{h,i}^-({\bf 0}) = 0, \; \forall i$.
\end{enumerate}

\end{theorem}
\begin{proof} (i) The decomposition outlined in Lemma \ref{lemma:decomp} and Corollary \ref{cor:NLPfactor} results in $1+n_{g_{NL}}+2n_{h_{NL}}$ families of constraints, with each family having the form of (\ref{eq:decompcon}). Because the approximation procedure is identical for each family,  only the approximation of the general form (\ref{eq:decompcon}) will be considered here. 

First, let $p^1_{\phi,ij}$ and $p^2_{\phi,ij}$ denote piecewise-concave underapproximations of the univariate functions $\phi_{ij}$ and $-\phi_{ij}$, respectively. These approximations must exist in view of Lemma \ref{lemma:pwappuni}, Lemma \ref{lemma:underover}, and the conditions imposed by Assumptions A1 and A2. The constraints 

\vspace{-2mm}
$$\phi_{ij} (y_j) - Z_{b,ij} \leq 0, \;\;\; -\phi_{ij} (y_j) + Z_{b,ij} \leq 0$$ 

\noindent are replaced by the relaxed constraints 

\vspace{-2mm}
$$
p_{\phi,ij}^1 (y_j) - Z_{b,ij} \leq 0, \;\;\;p_{\phi,ij}^2 (y_j) + Z_{b,ij} \leq 0.
$$

\noindent Because $\phi_{ij} (y_j) - \epsilon_p^- \leq p^1_{\phi,ij} (y_j)$ and $-\phi_{ij} (y_j) - \epsilon_p^- \leq p^2_{\phi,ij} (y_j)$ for all $y_j$ in $\mathcal{Y}$, it follows that

\vspace{-2mm}
\begin{equation}\label{eq:erreq}
\begin{array}{l}
\begin{array}{l}
p_{\phi,ij}^1 (y_j) - Z_{b,ij} \leq 0 \vspace{1mm} \\
p_{\phi,ij}^2 (y_j) + Z_{b,ij} \leq 0
\end{array} \Rightarrow
\begin{array}{l}
\phi_{ij} (y_j) - \epsilon_p^- - Z_{b,ij} \leq 0 \vspace{1mm} \\
-\phi_{ij} (y_j) - \epsilon_p^- + Z_{b,ij} \leq 0
\end{array} \vspace{1mm} \\
\hspace{40mm} \Leftrightarrow
\phi_{ij} (y_j) - \epsilon_p^- \leq Z_{b,ij} \leq  \phi_{ij} (y_j) + \epsilon_p^-.
\end{array}
\end{equation}

\noindent It is thus clear that $Z_{b,ij}$ is bounded by the range of $\phi_{ij}$ over $\mathcal{Y}$ and the approximation error, with the boundedness of $\phi_{ij}$ following from Assumptions A1 and A2. Let $\underline Z_{b,ij}, \overline Z_{b,ij}$ denote the corresponding bounds.

From the equality $Z_{c,i(n_y-1)} -  Z_{b,in_y} = 0$, it follows that $Z_{c,i(n_y-1)}$ may be bounded by $\underline Z_{c,i(n_y-1)} = \underline Z_{b,in_y}$ and $\overline Z_{c,i(n_y-1)} = \overline Z_{b,in_y}$. Proceeding up the list of constraints in (\ref{eq:decompcon}), the constraint pair 

\vspace{-2mm}
$$
Z_{c,ij} - Z_{b,i(j+1)} Z_{c,i(j+1)} \leq 0, \;\;\;Z_{b,i(j+1)} Z_{c,i(j+1)} - Z_{c,ij} \leq 0
$$ 

\noindent is now considered for $j = n_y-2$, i.e.:

\vspace{-2mm}
$$
Z_{c,i(n_y-2)} - Z_{b,i(n_y-1)} Z_{c,i(n_y-1)} \leq 0, \;\;\;Z_{b,i(n_y-1)} Z_{c,i(n_y-1)} - Z_{c,i(n_y-2)} \leq 0.
$$

\noindent Since $Z_{b,i(n_y-1)}$ and $Z_{c,i(n_y-1)}$ are both bounded, it follows that Lemma \ref{lemma:pwappbilinear} may be applied to yield $\epsilon_p^-$-accurate underapproximations of both $- Z_{b,i(n_y-1)} Z_{c,i(n_y-1)}$ and $Z_{b,i(n_y-1)} Z_{c,i(n_y-1)}$ -- denoted by $p^1_{i(n_y-1)}$ and $p^2_{i(n_y-1)}$, respectively. Replacing the constraints by the approximations yields:

\vspace{-2mm}
$$
Z_{c,i(n_y-2)} + p^1_{i(n_y-1)} (Z_{b,i(n_y-1)}, Z_{c,i(n_y-1)}) \leq 0, \;\;\; p^2_{i(n_y-1)} (Z_{b,i(n_y-1)}, Z_{c,i(n_y-1)}) - Z_{c,i(n_y-2)} \leq 0.
$$ 

\noindent Noting that 

\vspace{-2mm}
$$
\begin{array}{l}
- Z_{b,i(n_y-1)} Z_{c,i(n_y-1)} - \epsilon_p^- \leq p^1_{i(n_y-1)} (Z_{b,i(n_y-1)}, Z_{c,i(n_y-1)}) \vspace{1mm} \\ 
Z_{b,i(n_y-1)} Z_{c,i(n_y-1)} - \epsilon_p^- \leq p^2_{i(n_y-1)} (Z_{b,i(n_y-1)}, Z_{c,i(n_y-1)})
\end{array}
$$ 

\noindent for all possible $(Z_{b,i(n_y-1)}, Z_{c,i(n_y-1)})$ in the box defined by $\underline Z_{b,i(n_y-1)} \leq Z_{b,i(n_y-1)} \leq \overline Z_{b,i(n_y-1)}$ and $\underline Z_{c,i(n_y-1)} \leq Z_{c,i(n_y-1)} \leq \overline Z_{c,i(n_y-1)}$ thus leads to the implication

\vspace{-2mm}
$$
\begin{array}{l}
\begin{array}{l}
Z_{c,i(n_y-2)} + p^1_{i(n_y-1)} (Z_{b,i(n_y-1)}, Z_{c,i(n_y-1)}) \leq 0 \vspace{1mm} \\
p^2_{i(n_y-1)} (Z_{b,i(n_y-1)}, Z_{c,i(n_y-1)}) - Z_{c,i(n_y-2)} \leq 0
\end{array} \Rightarrow
\begin{array}{l}
Z_{c,i(n_y-2)} - Z_{b,i(n_y-1)} Z_{c,i(n_y-1)} - \epsilon_p^- \leq 0 \vspace{1mm} \\
Z_{b,i(n_y-1)} Z_{c,i(n_y-1)} - \epsilon_p^- - Z_{c,i(n_y-2)} \leq 0
\end{array} \vspace{2mm} \\ 
\hspace{35mm}\Leftrightarrow Z_{b,i(n_y-1)} Z_{c,i(n_y-1)} - \epsilon_p^- \leq Z_{c,i(n_y-2)} \leq  Z_{b,i(n_y-1)} Z_{c,i(n_y-1)} + \epsilon_p^-.
\end{array}
$$

\noindent Because the bilinear terms are bounded, it follows that $Z_{c,i(n_y-2)}$ be bounded as well, with $\underline Z_{c,i(n_y-2)}, \overline Z_{c,i(n_y-2)}$ used to denote the bounds.

Repeating this procedure another $n_y-1$ times for $j = n_y-3,...,1$, one can prove the existence of the analogous bounds $\underline Z_{c,ij} \leq Z_{c,ij} \leq \overline Z_{c,ij}$ and the analogous approximations 

\vspace{-2mm}
$$
Z_{c,ij} + p^1_{i(j+1)}(Z_{b,i(j+1)}, Z_{c,i(j+1)}) \leq 0, \;\;\; p^2_{i(j+1)}(Z_{b,i(j+1)}, Z_{c,i(j+1)}) - Z_{c,ij} \leq 0.
$$ 

Finally, one arrives at the constraints 

\vspace{-2mm}
$$
Z_{b,11} Z_{c,11} + \sum_{i=1}^{m-1} z_{a,i} \leq 0, \;\;\;Z_{b,i1} Z_{c,i1} - z_{a,i-1} \leq 0,
$$ 

\noindent which may be underapproximated by  

\vspace{-2mm}
$$
p_{11} (Z_{b,11}, Z_{c,11}) + \sum_{i=1}^{m-1} z_{a,i} \leq 0, \;\;\;p_{i1}(Z_{b,i1}, Z_{c,i1}) - z_{a,i-1} \leq 0,
$$

\noindent the existence of $p_{i1}$ for $i = 1,...,m$ following from the boundedness of $Z_{b,i1}$ and $Z_{c,i1}$. 

Noting that all of the nonlinear constraints of (\ref{eq:decompcon}) have been replaced by piecewise-concave approximations, the general equivalence

\vspace{-2mm}
$$
\mathop {\max} \limits_{i} f_i (x) \leq 0 \; \Leftrightarrow f_i (x) \leq 0, \;\; \forall i
$$

\noindent may now be used to transform the constraint set into an equivalent set of concave constraints. 

(ii) To prove that $f_{NL} (y_{app}^*) \leq f_{NL} (y^*)$, it is first established that the feasibility of (\ref{eq:mainprobeq}) implies the feasibility of the RCP approximation. Letting $(\tilde y, \tilde t)$ denote a feasible point of (\ref{eq:mainprobeq}), consider first the decomposed problem prior to approximation, with the corresponding auxiliary variables implicitly fixed as $\tilde Z_{b,ij} = \phi_{ij} (\tilde y_j)$. By recursively exploiting the relation $Z_{b,i(j+1)} Z_{c,i(j+1)} = Z_{c,ij}$, one readily obtains the implication

\vspace{-2mm}
$$
\begin{array}{l}
\displaystyle \prod_{j=1}^{n_y} \phi_{1j} (\tilde y_j) + \sum_{i=1}^{m-1} z_{a,i} \leq 0 \vspace{1mm} \\
\displaystyle \prod_{j=1}^{n_y} \phi_{ij} (\tilde y_j) - z_{a,i-1} \leq 0, \;\; i = 2,...,m
\end{array} \Rightarrow f(\tilde y) \leq 0,
$$

\noindent which then implies that $g_{NL,i}(\tilde y) \leq 0$, $h_{NL,i}(\tilde y) \leq 0$, and $-h_{NL,i}(\tilde y) \leq 0$, $\forall i$. Since the epigraph variable $t$ may be added to the implication above without affecting the validity of the derivation, the inequality $f_{NL} (\tilde y) - \tilde t \leq 0$ is satisfied as well. Because the RCP approximation is obtained by relaxing the constraints via underapproximations, it follows that there would exist a set of auxiliary variables $z_{all} = \tilde z_{all}$ such that the set $(\tilde y, \tilde t, \tilde z_{all})$ be feasible for the approximation -- the choices dictated by $\tilde z_{a_,i-1} = \prod_{j=1}^{n_y} \phi_{ij}(\tilde y_j)$, $\tilde Z_{b,ij} = \phi_{ij} (\tilde y_j)$, $\tilde Z_{b,i(j+1)} \tilde Z_{c,i(j+1)} = \tilde Z_{c,ij}$, and $\tilde Z_{c,i(n_y-1)} = \tilde Z_{b,in_y}$ being one example.

Because $(y^*, t^*) = (y^*, f_{NL}(y^*))$ is feasible for (\ref{eq:mainprobeq}), it follows that there exists a set of auxiliary variables $z_{all} = z_{all}^*$ such that the decision-variable set $(y^*, f_{NL}(y^*), z_{all}^*)$ is feasible for the RCP approximation. This point has a cost function value of $f_{NL}(y^*)$, which ensures that $f_{NL}(y_{app}^*) \leq f_{NL}(y^*)$ since the global optimum of the RCP approximation can only improve upon this value.

(iii) Finally, it remains to prove that the potential constraint violations when applying $y_{app}^*$ to the original problem are bounded by finite-degree polynomials of the approximation errors. So as to avoid long, messy expressions, \emph{any} finite-degree polynomial of the vector of errors $\zeta$ with a zero at the origin will be expressed as $q(\zeta)$. Because such polynomials are closed under addition and multiplication, they will be manipulated, for convenience, in the equivalence sense -- e.g., $q(\zeta) \equiv q(\zeta)q(\zeta) \equiv a_s q(\zeta)$, with $a_s \in \mathbb{R}$. The set of all such polynomials will be denoted by $\mathcal{Q}$.

From (\ref{eq:erreq}), it follows that the auxiliary variable $Z_{b,ij}^*$ at the global solution of the RCP approximation must satisfy 

\vspace{-2mm}
$$
Z_{b,ij}^* = \phi_{ij} (y_{app,j}^*) + \zeta_p.
$$

\noindent Since $\zeta_p \in \mathcal{Q}$, this may be equivalently restated as

\vspace{-2mm}
$$
Z_{b,ij}^* \equiv \phi_{ij} (y_{app,j}^*) + q(\zeta).
$$

Following the same recursive procedure as before, note that $Z_{c,i(n_y-1)}^* \equiv \phi_{in_y} (y_{app,n_y}^*) + q(\zeta)$. Starting the chain of implications, one has that

\vspace{-2mm}
$$
\begin{array}{rcl}
Z_{b,i(n_y-1)}^* Z_{c,i(n_y-1)}^* & \equiv &  \left[ \phi_{i(n_y-1)} (y_{app,n_y-1}^*) + q(\zeta) \right] \left[ \phi_{in_y} (y_{app,n_y}^*) + q(\zeta) \right] \vspace{1mm} \\
 & \equiv & \phi_{i(n_y-1)} (y_{app,n_y-1}^*) \phi_{in_y} (y_{app,n_y}^*) + q(\zeta),
\end{array}
$$

\noindent since $\phi_{in_y} (y_{app,n_y}^*)q(\zeta)$, $\phi_{i(n_y-1)} (y_{app,n_y-1}^*) q(\zeta)$, and  $q(\zeta)q(\zeta)$ all belong to $\mathcal{Q}$, and thus so does their sum. Because

\vspace{-2mm}
$$
Z_{b,i(n_y-1)} Z_{c,i(n_y-1)} - \epsilon_p^- \leq Z_{c,i(n_y-2)} \leq  Z_{b,i(n_y-1)} Z_{c,i(n_y-1)} + \epsilon_p^-
$$

\noindent and $q(\zeta) + \zeta_p \in \mathcal{Q}$, it follows that

\vspace{-2mm}
$$
Z_{c,i(n_y-2)}^* = Z_{b,i(n_y-1)}^* Z_{c,i(n_y-1)}^* + \zeta_p \equiv \phi_{i(n_y-1)} (y_{app,n_y-1}^*) \phi_{in_y} (y_{app,n_y}^*) + q(\zeta).
$$

Generalizing the recursion, it is easy to show that

\vspace{-2mm}
$$
Z_{c,ij}^* \equiv \prod_{k = j+1}^{n_y} \phi_{i k} (y_{app,k}^*) + q(\zeta),
$$

\noindent which then allows the implications

\vspace{-2mm}
$$
\begin{array}{l}
\begin{array}{l}
\displaystyle Z_{b,11}^* Z_{c,11}^* + \sum_{i=1}^{m-1} z_{a,i} \leq 0 \vspace{1mm} \\
\displaystyle Z_{b,i1}^* Z_{c,i1}^* - z_{a,i-1} \leq 0, \;\; i = 2,...,m
\end{array} \equiv
\begin{array}{l}
\displaystyle \prod_{j = 1}^{n_y} \phi_{1j} (y_{app,j}^*) + q(\zeta) + \sum_{i=1}^{m-1} z_{a,i} \leq 0 \vspace{1mm} \\
\displaystyle \prod_{j = 1}^{n_y} \phi_{ij} (y_{app,j}^*) + q(\zeta) - z_{a,i-1} \leq 0, \;\; i = 2,...,m
\end{array} \vspace{1mm} \\ 
\hspace{55mm} \Rightarrow f(y_{app}^*) + mq(\zeta) \leq 0 \equiv f(y_{app}^*) \leq q(\zeta).
\end{array}
$$

The final result, $f(y_{app}^*) \leq q(\zeta)$, yields $g_{NL,i} (y_{app}^*) \leq q_{g,i}(\zeta_{g_i})$ for $f = g_{NL,i}$, $q = q_{g,i}$, and $\zeta = \zeta_{g_i}$. The results $h_{NL,i} (y_{app}^*) \leq q_{h,i}^+ (\zeta_{h_i}^+)$ and $-h_{NL,i} (y_{app}^*) \leq q_{h,i}^- (\zeta_{h_i}^-)$ follow in the same manner, and in turn imply that $|h_{NL,i} (y_{app}^*)| \leq \mathop {\max} \left( q_{h,i}^+ (\zeta_{h_i}^+), q_{h,i}^- (\zeta_{h_i}^-) \right)$. \qed

\end{proof}

From Theorem \ref{thm:mainapprox}, one sees that it is possible to make the overall errors, as described by the polynomials $q$, arbitrarily small as $\epsilon_p^- \rightarrow 0$. This is due to $\epsilon_p^- \rightarrow 0 \Rightarrow \zeta_p \rightarrow 0$ and the fact that the polynomials are Lipschitz-continuous with zero values at the origin.

\subsection{A Basic Piecewise-Concave Approximation Algorithm}

An approximation algorithm for univariate functions that is consistent with Lemma \ref{lemma:pwappuni} is now proposed. This algorithm is largely based on the proof of the lemma as given in \cite{Bunin:2014}, and may be qualitatively outlined as follows:

\begin{enumerate}
\item The approximation interval is discretized into two evenly spaced grids. One of the discretizations is coarse and generates wider subintervals intended for the corresponding pieces of the piecewise-concave function. The other discretization is fine and is used to obtain a relatively tight bound on the error of the approximation.
\item Concave parabolas that intersect the approximated function at the midpoints of their respective (coarse) subintervals are constructed. As an additional restriction, the parabolas are forced to have sufficiently large curvature -- enforced via constraints on their derivatives at the subinterval edges -- so as to ensure that each piece of the piecewise-concave function is only maximal over a single interval.
\item The piecewise maximum of the parabolas is shifted up or down to guarantee an over- or underapproximation, with the shift corresponding to the error bound as computed for the fine discretization.
\end{enumerate}

\begin{algorithm}[A piecewise-concave approximation by parabolas]\label{algo:pwapp}

\noindent {\bf User input}: $\phi : \mathbb{R} \rightarrow \mathbb{R}$, $\kappa$, $n_{p}$, $n_{p}^{fine}$ (a large integer satisfying $n_{p}^{fine} >\hspace{-1mm}> n_p$), $\underline x_j$, $\overline x_j$, $\rho$ (equal to 1 if an overapproximation is desired and to $-1$ if an underapproximation).

\noindent {\bf Output}: $\beta_2, \beta_1, \beta_0 \in \mathbb{R}^{n_p}$, with $p(x_j) = \displaystyle \mathop {\max}_{i=1,...,n_p} \left( \beta_{2,i} x_j^2 + \beta_{1,i} x_j + \beta_{0,i} \right) \approx \phi(x_j)$ over $[\underline x_j, \overline x_j]$.

\begin{enumerate}

\item (Discretizations) Define the discretization interval $\Delta x = (\overline x_j - \underline x_j)/n_p$ and the discrete coordinate set $x_d = \{ \underline x_j, \underline x_j + \Delta x, ..., \overline x_j - \Delta x, \overline x_j \}$. Define a finer version of both as $\Delta x^{fine} = (\overline x_j - \underline x_j)/n_{p}^{fine}$ and $x_{d}^{fine} = \{ \underline x_j, \underline x_j + \Delta x^{fine}, ..., \overline x_j - \Delta x^{fine}, \overline x_j \}$.
\item (Calculating parabola coefficients) For $i = 1,...,n_p$, compute the coefficients $\beta_{2,i}, \beta_{1,i}, \beta_{0,i}$ as

\vspace{-2mm}
$$
\begin{array}{l}
\left[ {\begin{array}{*{20}c}
   \beta_{2,i}  \\
   \beta_{1,i} \\
   \beta_{0,i}
\end{array}} \right] := 
\left[ {\begin{array}{ccccc}
   (x_{d,i} + 0.5 \Delta x)^2 & & x_{d,i} + 0.5 \Delta x & & 1  \\
   2 x_{d,i} & & 1 & & 0 \\
   2 x_{d,i+1} & & 1 & & 0
\end{array}} \right] ^{-1} 
\left[ {\begin{array}{*{20}c}
   \phi(x_{d,i} + 0.5 \Delta x)  \\
   2 \kappa \\
   -2 \kappa
\end{array}} \right]
\end{array}.
$$

\item (Shift to ensure under/overapproximation) For $i = 1,...,n_{p}^{fine}+1$, compute the approximation errors as $\bar \epsilon_{p,i} := p(x_{d,i}^{fine}) - \phi(x_{d,i}^{fine})$ and bound the worst-case error as

\vspace{-2mm}
$$
\mathop {\min}_{i = 1,...,n_p^{fine}} \bar \epsilon_{p,i} - 2.5 \kappa \Delta x^{fine} < p(x_{j}) - \phi(x_{j}) < \mathop {\max}_{i = 1,...,n_p^{fine}} \bar \epsilon_{p,i} + 2.5 \kappa \Delta x^{fine}.
$$

\noindent If $\rho = 1$, set $\beta_{0,i} := \beta_{0,i} + 2.5 \kappa \Delta x^{fine} - \displaystyle \mathop {\min}_{j = 1,...,n_p^{fine}} \bar \epsilon_{p,j}, \; \forall i = 1,...,n_p$. If $\rho = -1$, set $\beta_{0,i} := \beta_{0,i} - 2.5 \kappa \Delta x^{fine} - \displaystyle \mathop {\max}_{j = 1,...,n_p^{fine}} \bar \epsilon_{p,j}, \; \forall i = 1,...,n_p$.

\end{enumerate}

\end{algorithm}

To make a clean link with the result of Lemma \ref{lemma:pwappuni}, note that to obtain the error bound (\ref{eq:pwapp}) provided a certain $\epsilon_p$, it is sufficient to set $n_p$ such that $\Delta x \leq \epsilon_p/(2.5 \kappa)$ \cite{Bunin:2014}.

A demonstration of the algorithm when applied to a sinusoidal function is given in Fig.~\ref{fig:pwapp}, where one sees that increasing the number of pieces eventually leads to the function being approximated almost perfectly. Though computationally light and theoretically rigorous, the algorithm is not very efficient as it does not attempt to adapt to the innate local concave structures of the approximated function -- instead, it simply constructs parabolas that become more and more needlelike as $n_p$ is increased. More efficient alternatives are certainly possible, but are outside the scope of the present work. 

For the case when the approximated function is convex, as occurs when approximating decomposed bilinear terms, the procedure of gridding and obtaining a piecewise-linear approximation is used. For the composite case of $f(a^T x + b)$, with $f$ convex and $a \in \mathbb{R}^n$, $b \in \mathbb{R}$, this is done by adding an auxiliary variable and the linear equality $z = a^T x + b$, and then approximating the univariate $f(z)$ so as to avoid gridding in higher dimensions.

\begin{figure}
\begin{center}
\includegraphics[width=.9\textwidth]{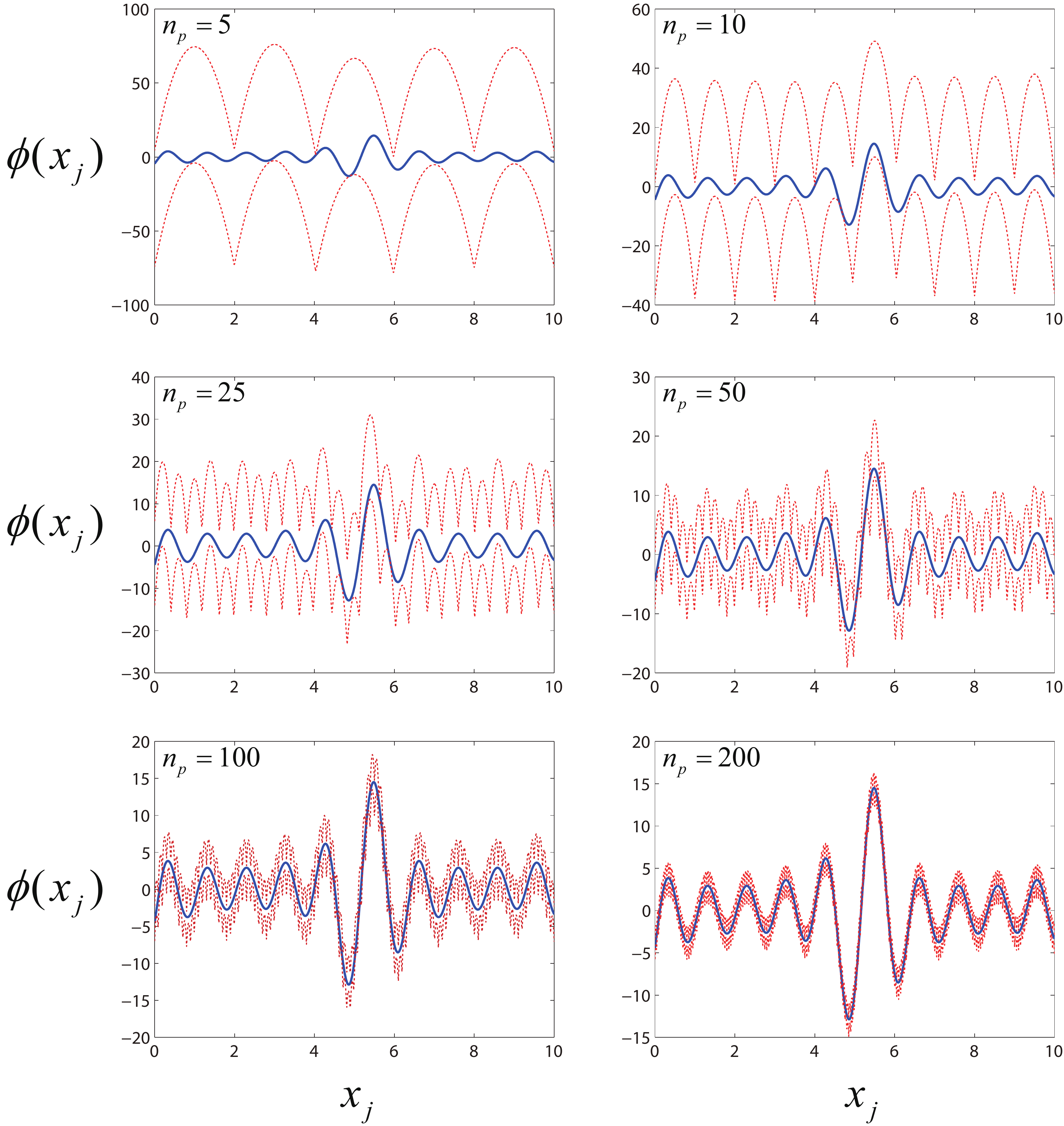}
\caption{Under/overapproximations (dashed lines) of $\phi(x_j) = \sum_{i = 1}^{5} i \; {\rm cos} \left( (i+1)x_j+i \right)$ (solid line) on the interval $x_j \in [0,10]$ as constructed by Algorithm \ref{algo:pwapp}. The Lipschitz constant is taken as $\kappa = 70$.}
\label{fig:pwapp}
\end{center}
\end{figure}

\subsection{A Decomposition and Approximation Example} 

The decomposition procedure of Lemma~\ref{lemma:decomp} together with the approximation steps of Theorem~\ref{thm:mainapprox} are sufficient to obtain an arbitrarily good RCP approximation of any factorable NLP problem that satisfies Assumptions A1 and A2. However, such a procedure is usually not necessary and may not be the most efficient in practice. As two examples of where inefficiency could arise, note that (a) there is no need to break an equality constraint that is linear and that (b) there is no need to decompose an inequality constraint function that is already concave. The following example illustrates how one may decompose and approximate a given NLP problem using a more practical approach.  

\vspace{2mm}
\noindent {\it Example 1 (RCP approximation of a factorable NLP problem)}
\vspace{2mm}

Consider the problem

\vspace{-2mm}
\begin{equation}\label{eq:sinaff}
\begin{array}{rl}
\mathop {\rm minimize}\limits_{y_1,y_2} & (\mathop {\sin} y_1)(-y_1 + 0.3y_2)  \\
 {\rm{subject}}\;{\rm{to}} & y_1 \in [-2,2], \; y_2 \in [-5,5].
\end{array}
\end{equation}

The auxiliary variable $t$ is added to obtain a linear cost via the epigraph transformation:

\vspace{-2mm}
$$
\begin{array}{rl}
\mathop {\rm minimize}\limits_{y_1,y_2,t} & t  \\
 {\rm{subject}}\;{\rm{to}} & (\mathop {\sin} y_1)(-y_1 + 0.3y_2) - t \leq 0 \\
 & y_1 \in [-2,2], \; y_2 \in [-5,5].
\end{array}
$$

Introducing the additional variables $z_1 = \mathop {\sin} y_1$ and $z_2 = -y_1 + 0.3y_2$ leads to

\vspace{-2mm}
$$
\begin{array}{rl}
\mathop {\rm minimize}\limits_{y_1,y_2,t,z_1,z_2} & t  \\
 {\rm{subject}}\;{\rm{to}} & z_1 z_2 - t \leq 0 \\
& \mathop {\sin} y_1 - z_1 = 0 \\
 & y_1 \in [-2,2], \; y_2 \in [-5,5] \\
 & -y_1 + 0.3y_2 - z_2 = 0,
\end{array}
$$

\noindent where it is noted that $z_1 \in [-1, 1]$ and $z_2 \in [-3.5, 3.5]$ implicitly.

The inequality constraint $z_1 z_2 - t \leq 0$ is then decomposed as follows:

\vspace{-2mm}
$$
\begin{array}{lll}
z_1 z_2 - t \leq 0 & \Leftrightarrow & 0.5(z_1 + z_2)^2 - 0.5 z_1^2 - 0.5 z_2^2 - t \leq 0 \vspace{2mm} \\
 & \Leftrightarrow & \begin{array}{l} 0.5z_3^2 - 0.5 z_1^2 - 0.5 z_2^2 - t \leq 0 \\ z_3 - z_1 - z_2 = 0, \end{array} \vspace{2mm} \\
 & \displaystyle \mathop {\Leftrightarrow}^{\epsilon_p^- \rightarrow 0} & \begin{array}{l} 0.5 p_{a,i}(z_3) - 0.5 z_1^2 - 0.5 z_2^2 - t \leq 0, \;\; i = 1,...,n_{p,a} \\ z_3 - z_1 - z_2 = 0,   \end{array}
\end{array}
$$

\noindent with $p_{a}$ an $n_{p,a}$-piece piecewise-concave approximation of $z_3^2$ over the interval $[-4.5, 4.5]$, as $z_3 \in [-4.5, 4.5]$ follows implicitly from the bounds on $z_1$ and $z_2$.

The equality constraint $\mathop {\sin} y_1 - z_1 = 0$ is broken and decomposed as

\vspace{-2mm}
\begin{equation}\label{eq:decomp2}
\begin{array}{lll}
\mathop {\sin} y_1 - z_1 = 0 & \Leftrightarrow & \begin{array}{l} \mathop {\sin} y_1 - z_1 \leq 0 \\ -\mathop {\sin} y_1 + z_1 \leq 0 \end{array} \vspace{2mm} \\
& \displaystyle \mathop {\Leftrightarrow}^{\epsilon_p^- \rightarrow 0} &   \begin{array}{l} p_{b,i} (y_1) - z_1 \leq 0, \;\; i = 1,...,n_{p,b} \\ p_{c,i}( y_1) + z_1 \leq 0, \;\; i = 1,...,n_{p,c}, \end{array}
\end{array}
\end{equation}

\noindent where $p_{b}$ and $p_c$ are $n_{p,b}$- and $n_{p,c}$-piece piecewise-concave approximations, over the interval $y_1 \in [-2, 2]$, of $\mathop {\sin} y_1$ and $-\mathop {\sin} y_1$, respectively.

Defining $x = (y_1,y_2,t,z_1,z_2,z_3)$, the standard-form RCP approximation of (\ref{eq:sinaff}) follows:

\vspace{-2mm}
\begin{equation}\label{eq:sinaffRCP}
\begin{array}{rl}
\mathop {\rm minimize}\limits_{x} & x_3  \\
 {\rm{subject}}\;{\rm{to}} & 0.5 p_{a,i}(x_6) - 0.5 x_4^2 - 0.5 x_5^2 - x_3 \leq 0, \;\; i = 1,...,n_{p,a} \\
& p_{b,i} (x_1) - x_4 \leq 0, \;\; i = 1,...,n_{p,b} \\
& p_{c,i} (x_1) + x_4 \leq 0, \;\; i = 1,...,n_{p,c} \\
& -x_1 - 2 \leq 0, \; x_1 - 2 \leq 0 \\
& -x_2 - 5 \leq 0, \; x_2 - 5 \leq 0 \\
 & -x_1 + 0.3x_2 - x_5 = 0 \\
 & - x_4 - x_5 + x_6 = 0.
\end{array}
\end{equation}

\section{RCP Regularity}
\label{sec:RCP}

It has been previously shown \cite{Ueing:72,Hillestad:80} that Problem (\ref{eq:RCPprob}) may be solved reliably by an enumeration approach for the case where the cost function is not linear but strictly concave. In particular, such a method exploits the properties that every local minimum (a) must occur at an intersection of $n$ constraints whose gradients at the minimum are linearly independent, and (b) may be found by solving the convex ``reverse problem'' (\ref{eq:RCPprobrevgen}). It thus follows that one may enumerate all of the $\binom{n_g}{n-n_C}$ inequality constraint combinations, solve the reverse problem for each, and then compare the cost values for the resulting feasible points.

As pointed out by Hillestad and Jacobsen \cite{Hillestad:80}, such an enumeration procedure fails to generalize to problems with a nonstrict concave cost function since the second property need not hold. Because the strict concavity assumption may be too restrictive for some practical problems, an attempt to relax this restriction is made here.

\begin{definition}[RCP regularity]\label{def:RCPreg} The RCP problem (\ref{eq:RCPprob}) is called ``regular'' if (\ref{eq:RCPprobrev}) holds.
\end{definition}

Clearly, any regular RCP problem may be solved by enumeration, as solving (\ref{eq:RCPprobrevgen}) for all possible candidate sets $i_A$ must yield $x^*$ as the solution when $i_A = i_A^*$. The question of whether a given RCP problem is regular or not is open to further investigation. Here, a fairly general sufficient condition for RCP regularity that may be verified for certain problems is provided.

\begin{theorem}[Sufficient condition for RCP regularity]\label{thm:suffcon} If the global minimum $x^*$ is a strict local minimum of (\ref{eq:RCPprob}), then the RCP problem (\ref{eq:RCPprob}) is regular.
\end{theorem}
\begin{proof} Consider Problem (\ref{eq:RCPprob}) following linearization around $x^*$:

\vspace{-2mm}
\begin{equation}\label{eq:RCPproblin}
\begin{array}{rl}
\mathop {\rm minimize}\limits_{x} & c^T x \\
{\rm{subject}}\;{\rm{to}} & g_i(x^*) + \nabla g_i (x^*)^T (x - x^*) \leq 0,\hspace{3mm}i = 1,...,n_g \\
 & Cx = d.
\end{array}
\end{equation} 

\noindent Because $x^*$ is a strict local minimum of (\ref{eq:RCPprob}), it follows that it solves (\ref{eq:RCPproblin}) uniquely \cite[Theorem 6]{Hillestad:80}. Since $g_i (x^*) = 0, \; \forall i \in i_{A}^*$ and those constraints that are inactive at $x^*$ may be removed without affecting the solution, one has that $x^*$ must also uniquely solve

\vspace{-2mm}
\begin{equation}\label{eq:RCPproblinA}
\begin{array}{rl}
\mathop {\rm minimize}\limits_{x} & c^T x \\
{\rm{subject}}\;{\rm{to}} & \nabla g_i (x^*)^T (x - x^*) \leq 0,\hspace{3mm}i \in i_{A^*} \\
 & Cx = d.
\end{array}
\end{equation} 

To prove that $\# i_{A^*} = n - n_C$, first note that $\# i_{A^*} > n - n_C$ is impossible due to Assumption B4, as this would imply linear dependence of the gradients of the active constraints. That $\# i_{A^*} < n - n_C$ cannot be true follows from the property that $x^*$ must be a basic solution of (\ref{eq:RCPprob}), which constrains the Jacobian of the active set to have rank equal to $n$ \cite{Hillestad:80}.

The vector $x^*$ must also uniquely solve the reverse linearized problem

\vspace{-2mm}
\begin{equation}\label{eq:RCPproblinrev}
\begin{array}{rl}
\mathop {\rm maximize}\limits_{x} & c^T x \\
{\rm{subject}}\;{\rm{to}} & \nabla g_i (x^*)^T (x - x^*) \geq 0,\hspace{3mm}\forall i \in i_{A^*} \\
 & Cx = d,
\end{array}
\end{equation} 

\noindent as (\ref{eq:RCPproblinA}) and (\ref{eq:RCPproblinrev}) have the same optimality conditions.

Finally, as (\ref{eq:RCPproblinrev}) is the linearization of (\ref{eq:RCPprobrev}) around $x^*$, it follows from the convexity of (\ref{eq:RCPprobrev}) that the latter has a smaller feasible region than (\ref{eq:RCPproblinrev}), which implies that $x^*$ must solve (\ref{eq:RCPprobrev}) uniquely as well. The regularity of (\ref{eq:RCPprob}) then follows from Definition \ref{def:RCPreg}. \qed 

\end{proof}

A geometric illustration of the proof is given in Fig. \ref{fig:th3pic}. It is not difficult to see that RCP problems with a strict concave cost must only admit strict local minima and that this property is retained if the problem is placed into standard form following an epigraph transformation such as the one in (\ref{eq:mainprobeq}).

\begin{figure*}
\begin{center}
\includegraphics[width=.5\textwidth]{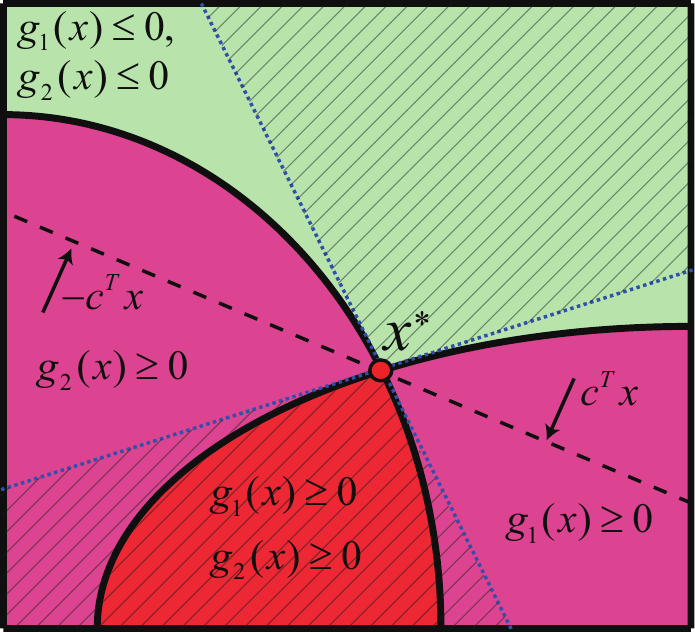}
\caption{Geometrical illustration of Theorem \ref{thm:suffcon} and its proof, with the light and dark areas denoting the feasible and infeasible spaces, respectively, of an RCP problem whose global minimum is defined by two concave inequality constraints. The shaded regions represent the feasible spaces of the linearized problems. Clearly, as $x^*$ is a strict local minimum for the original problem, it is also the strict local minimum for the original problem linearized around $x^*$. It is also clear that the linearized reverse problem has the same solution as the linearized original problem, and, as the infeasible-side linearized set is a superset of $g_i (x) \geq 0, \; \forall i \in i_{A^*}$, it follows that solving the reverse problem over the latter also yields $x^*$.}
\label{fig:th3pic}
\end{center}
\end{figure*}

\section{An Extended RCP Enumeration Framework}
\label{sec:algorithm}

The basic framework for an extended RCP enumeration-based solver is now presented in three parts. The enumeration scheme is outlined first, as this constitutes the base of the method and guarantees its finite-time convergence to a global minimum of (\ref{eq:RCPprob}) provided that the problem is regular. The second part then goes through a number of techniques that may dramatically speed up the search by reducing the number of active sets to be checked. These are divided into (a) techniques specific to (extended) RCP and (b) domain reduction techniques that are common to existing global optimization solvers. Finally, the third part combines these ideas to provide the RCP algorithm as it is coded in this work.

\subsection{Enumeration of Active Sets: General Procedure}
\label{sec:algo1}

The global minimum of a regular RCP problem may be found by simply checking all of the possible $\binom {n_g}{n-n_C}$ active sets, solving the corresponding convex problem (\ref{eq:RCPprobrevgen}) for each, and then comparing the solutions to find the one with the best value. However, much like the brute vertex-enumeration techniques of concave minimization \cite{McKeown:76,Pardalos:86}, such an approach, though guaranteed to solve the problem in finite time, is not computationally enviable when $\binom {n_g}{n-n_C}$ is large. The work by Ueing \cite{Ueing:72} has proposed considering active \emph{subsets} as a means of bypassing this difficulty -- i.e., by solving versions of (\ref{eq:RCPprobrevgen}) with less than $n-n_C$ inequality constraints used in the definition of $i_A$. This is essentially the approach pursued here, with the key difference being that the feasible domain of (\ref{eq:RCPprobrevgen}) is further restricted by the addition of the constraint $x \in \mathcal{C}$, where $\mathcal{C}$ is some convex set known to contain $x^*$ (e.g., a convex relaxation of $\mathcal{X}$).

Envisioning the possible active sets as the ultimate level of a tree (Fig. \ref{fig:branchtree}), the basic idea of checking active subsets involves starting at the base level of the tree, with a single inequality constraint, and building up to the final level of $n - n_C$ constraints. For the example of Fig. \ref{fig:branchtree}, the method proposed here would check the subsets, represented by their constraint indices, in the order $\{ 1 \}$, $\{ 2 \}$, $\{ 3 \}$, $\{ 1, 2 \}$, $\{ 1, 3 \}$, $\{ 1, 4 \}$, $\{ 2, 3 \}$, $\{ 2, 4 \}$, $\{ 3, 4 \}$, $\{ 1, 2, 3 \}$, $\{ 1, 2, 4 \}$, $\{ 1, 2, 5 \}$, $\{ 1, 3, 4 \}$, $\{ 1, 3, 5 \}$, $\{ 1, 4, 5 \}$, $\{ 2, 3, 4 \}$, $\{ 2, 3, 5 \}$, $\{ 2, 4, 5 \}$, $\{ 3, 4, 5 \}$, with the resulting full active sets of cardinality $n - n_C$ checked last. Subsets such as $\{ 1,6 \}$ would not be considered because this branch cannot be ``grown'' into a full active set while maintaining the imposed numerical ordering. While such an enumeration potentially requires more computations than brute enumeration, much useful information may be discovered while checking the subsets, which makes it possible to remove certain collections of sets from consideration entirely, thus ultimately reducing the overall computational burden to far below that of brute enumeration.

This is now formalized in the following lemma.

\begin{figure*}
\includegraphics[width=1\textwidth]{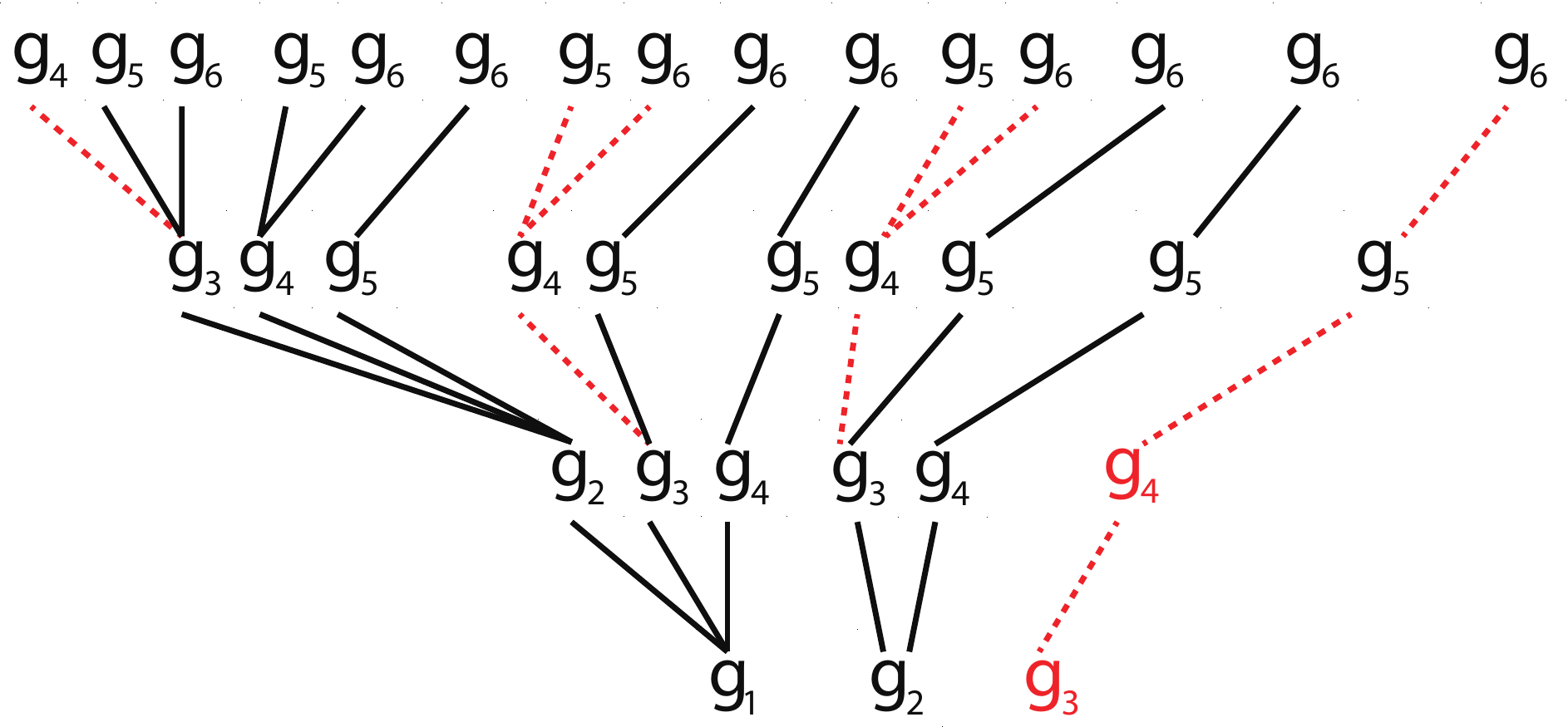}
\caption{The tree of possible active sets and subsets for a problem with $n - n_C = 4$ and $n_g = 6$. The effect of fathoming the subset $\{ 3, 4 \}$ is demonstrated, with all the resulting fathomed active sets shown via dotted branches.}
\label{fig:branchtree}
\end{figure*}

\begin{lemma}[Fathoming of active subsets]\label{lemma:fath1} Let $i_{\tilde A}$ denote the index set of $\tilde n < n - n_C$ inequality constraints of Problem (\ref{eq:RCPprob}), assumed regular with a global minimum $x^*$ and the corresponding active set $i_{A^*}$, and consider the problem

\vspace{-1mm}
\begin{equation}\label{eq:feascheck}
\begin{array}{rl}
\tilde x^* \in {\rm arg} \mathop {\rm minimize}\limits_{x} & c^T x \\
{\rm{subject}}\;{\rm{to}} & g_i(x) \geq 0, \;\; \forall i \in i_{\tilde A} \\
& Cx = d \\
& x \in \mathcal{C}.
\end{array}
\end{equation}

\noindent It follows that

\begin{enumerate}[(i)]
\item if (\ref{eq:feascheck}) is infeasible, then $i_{\tilde A} \not \subset i_{A^*}$,
\item if (\ref{eq:feascheck}) is feasible and $i_{\tilde A} \subset i_{A^*}$, then $c^T \tilde x^* \leq c^T x^*$,
\item defining $i_V = \{ i : g_i(\tilde x^*) \geq 0 \}$, Problem (\ref{eq:feascheck}) will be feasible for all $i_{\tilde A} \subseteq i_{V}$.
\end{enumerate}

\end{lemma}
\begin{proof}

\begin{enumerate}[(i)]
\item By contradiction, suppose that Problem (\ref{eq:feascheck}) is infeasible but that $i_{\tilde A} \subset i_{A^*}$. Because $g_i( x^* ) = 0, \; \forall i \in i_{A^*}$, it follows that $g_i( x^* ) = 0, \; \forall i \in i_{\tilde A}$. The equality constraints are also satisfied at $x^*$, while $x^* \in \mathcal{C}$ by definition. However, this contradicts the infeasibility of (\ref{eq:feascheck}), since $x^*$ is clearly feasible for this problem. 
\item Since $x^* \in \mathcal{C}$, the constraint $x \in \mathcal{C}$ may be added to (\ref{eq:RCPprobrev}) without affecting the solution:

\vspace{-1mm}
$$
\begin{array}{rl}
x^* = {\rm arg} \mathop {\rm maximize}\limits_{x} & c^T x \\
{\rm{subject}}\;{\rm{to}} & g_i(x) \geq 0, \;\; \forall i \in i_{A^*} \\
 & Cx = d \\
 & x \in \mathcal{C}.
\end{array}
$$

Consider now the minimization

\vspace{-1mm}
\begin{equation}\label{eq:feas3}
\begin{array}{rl}
\underline x^* \in {\rm arg} \mathop {\rm minimize}\limits_{x} & c^T x \\
{\rm{subject}}\;{\rm{to}} & g_i(x) \geq 0, \;\; \forall i \in i_{A^*} \\
& Cx = d \\
& x \in \mathcal{C},
\end{array}
\end{equation}

\noindent where it is clear that $c^T \underline x^* \leq c^T x^*$. Since (\ref{eq:feascheck}) is obtained from (\ref{eq:feas3}) by removing certain constraints, it follows that $c^T \tilde x^* \leq c^T \underline x^*$ and, consequently, that $c^T \tilde x^* \leq c^T x^*$.

\item Since $g_i (\tilde x^*) \geq 0, \; \forall i \in i_{V}$, it follows that $\tilde x^*$ will be a feasible point for all problems (\ref{eq:feascheck}) where $i_{\tilde A} \subseteq i_{ V}$. \qed

\end{enumerate}

\end{proof}

One sees that solving Problem (\ref{eq:feascheck}) is very useful as it does triple duty by

\begin{enumerate}[(a)]
\item fathoming subsets that cannot define an optimal active set whenever the problem is infeasible, thereby saving the computational effort of checking the active sets that include the fathomed combinations as subsets,
\item providing a lower bound on the globally optimal cost value potentially attained by certain active sets, thereby allowing for those sets to be fathomed if a feasible point with a cost value that is lower than this lower bound is found (i.e., $c^T x^* < c^T \tilde x^* \Rightarrow i_{\tilde A} \not \subset i_{A^*}$),
\item allowing for computational savings by foregoing Problem (\ref{eq:feascheck}) when it is known to be feasible due to $i_{\tilde A} \subseteq i_{V}$ for some previously found $i_V$.
\end{enumerate}

\noindent Of these three points, (a) is the most crucial as it allows for entire branches to be removed from the enumeration. Fig. \ref{fig:branchtree} illustrates this, where fathoming a single subset immediately removes 6 of the 15 possible active sets.

To properly manage the tree of active sets, \emph{fathoming} and \emph{validation} bases are built to cut out certain branches entirely (Points (a) and (b) above) and to skip solving (\ref{eq:feascheck}) for certain others (Point (c)), respectively. These will be denoted by the matrices $F$ and $V$, both of width $n_g$ with each column corresponding to one of the $n_g$ constraints. For algorithmic convenience, both $F$ and $V$ will be binary with 0 denoting the absence of a constraint and 1 denoting its presence. In the example of Fig. \ref{fig:branchtree}, a particular instance of $F$ and $V$ may be: 

\vspace{-1mm}
$$
F = \left[ \begin{array}{cccccc}
0 & 0 & 1 & 1 & 0 & 0 \\
1 & 1 & 1 & 0 & 0 & 0 \\
0 & 0 & 0 & 0 & 1 & 0  
\end{array}\right], \;\;
V = \left[ \begin{array}{cccccc}
1 & 0 & 0 & 1 & 0 & 1 \\
1 & 0 & 1 & 0 & 0 & 1  
\end{array}\right].
$$

\noindent For $F$, this may be read as stating that the index subsets $\{ 3, 4 \}, \{ 1, 2, 3 \}, \{ 5 \} \not \subset i_{A^*}$, thereby fathoming those $i_{\tilde A}$ and $i_A$ that are supersets of any of these sets. For $V$, this means that Problem (\ref{eq:feascheck}) will be feasible for any $i_{\tilde A} \subseteq \{ 1, 4, 6\}$ or $i_{\tilde A} \subseteq \{ 1, 3, 6\}$, and thus may be skipped for those subsets. With respect to the ordering of the constraints in $F$ and $V$, the convention used in this paper will be as follows:

\begin{itemize}
\item Any constraints that are known to be active at $x^*$ will be listed first.
\item The $2n$ constraints corresponding to the lower and upper bounds on the decision variables, which will be required for the algorithms presented here, will be listed in the last $2n$ columns of $F$ and $V$, with the last $n$ columns corresponding to the upper bounds and the $n$ columns before that corresponding to the lower.
\item The remaining constraints will be listed in between in the order specified by the user.
\end{itemize}

\subsection{Fathoming and Domain Reduction Techniques}
\label{sec:algo2}

While Lemma \ref{lemma:fath1} provides a basic means to fathom those active sets that cannot define $x^*$, a number of additional fathoming rules may be proposed and exploited to expedite the enumeration further. Some of these take direct advantage of the concavity, the active-set nature, and the approximation steps that characterize the extended RCP framework. Others, however, employ standard domain reduction techniques common to modern global optimization solvers so as to shrink $\mathcal{C}$ as much as possible, thereby increasing the number of subsets for which (\ref{eq:feascheck}) is infeasible.

\subsubsection{Fathoming Techniques Particular to RCP and Extended RCP}

\vspace{1mm}
\noindent {\it Innate Fathoming}
\vspace{1mm}

Once the RCP problem has been defined, there are certain constraints or constraint combinations that are innately known to either never intersect or to not define $x^*$. These are stated in the following lemma.

\begin{lemma}[Innate fathoming rules]\label{lemma:innatefathom} The following $i_{\tilde A}, i_A$ all satisfy $i_{\tilde A}, i_A \not \subseteq i_{A^*}$:

\begin{enumerate}[(i)]
\item Any $i_{\tilde A}, i_A$ with

\begin{equation}\label{eq:nofullrank}
{\rm rank}\; \left[ 
\begin{array}{c}
J_{\tilde A} \\
C \end{array} \right] < \tilde n + n_C, \;\;\;
{\rm rank}\; \left[ 
\begin{array}{c}
J_{A} \\
C \end{array} \right] < n + n_C
\end{equation}

\noindent where $J_{\tilde A}$ (resp., $J_A$) is the Jacobian, at $x^*$, of the constraints defined by $i_{\tilde A}$ (resp., $i_A$).

\item Any $i_{\tilde A}, i_A$ that includes the pair of indices corresponding to the constraints

\vspace{-2mm}
$$
p_i(x_l) + f(x) \leq 0, \;\; p_j(x_l) + f(x) \leq 0,
$$

\noindent with $i \neq j$ denoting the indices of different pieces of the piecewise-concave function $p(x_l) = \mathop {\max} \limits_{k = 1,...,n_p} p_k (x_l)$, univariate in $x_l$, where

\vspace{-2mm}
$$
\exists k \in \{ 1,...,n_p \} : p_i (x_l) = p_j(x_l) \Rightarrow p_k(x_l) + f(x) > 0,
$$

\noindent or

\vspace{-2mm}
$$
p_i (x_l) = p_j(x_l) \Rightarrow x_l \not \in [\underline x_l, \overline x_l].
$$

\noindent Here, $\underline x_l$ and $\overline x_l$ are defined as $\mathop {\min} \{ x_l : x \in \mathcal{X} \}$ and $\mathop {\max} \{ x_l : x \in \mathcal{X} \}$, respectively.

\item Any $i_{\tilde A}, i_A$ that includes the pair of indices corresponding to the constraints 

\vspace{-2mm}
$$
p_i^+ (x) + f_1(x) \leq 0, \;\;\;p_j^- (x) - f_1(x) \leq 0,
$$

\noindent where $\mathop {\max} \limits_{i = 1,...,n_p^+} p_i^+ (x) < -f_2(x)$ and $\mathop {\max} \limits_{j = 1,...,n_p^-} p_j^- (x) < f_2(x)$ for all $x \in \mathcal{X}$. Here, $p^-$ is the strict piecewise-concave underapproximation of the function $f_2$ while $p^+$ is the strict piecewise-concave underapproximation of its negative.

\end{enumerate}

\end{lemma}
\begin{proof}

\begin{enumerate}[(i)]
\item From Assumption B4, it is required that the Jacobian of the active inequality and equality constraints at $x^*$ have full rank, which is precluded by (\ref{eq:nofullrank}) as this implies linear dependence in some of the components.

\item Both constraints being active at $x^*$ implies $p_i(x_l^*) = p_j(x_l^*)$, which in turn either implies $p_k (x_l^*) + f(x^*) > 0$ or $x_l^* \not \in [\underline x_l, \overline x_l]$, both of which contradict the feasibility of $x^*$.

\item By contradiction, suppose that both $p_i^+ (x) + f_1(x) \leq 0$ and $p_j^- (x) - f_1(x) \leq 0$ are active at $x^*$:

\vspace{-2mm}
\begin{equation}\label{eq:spliteq}
\begin{array}{l}
p_i^+ (x^*) + f_1(x^*) = 0 \\
p_j^- (x^*) - f_1(x^*) = 0 
\end{array} \Rightarrow p_i^+ (x^*) + p_j^- (x^*) = 0.
\end{equation}

However, from $\mathop {\max} \limits_{i = 1,...,n_p^+} p_i^+ (x) < -f_2(x)$ and $\mathop {\max} \limits_{j = 1,...,n_p^-} p_j^- (x) < f_2(x)$, it follows that

$$
\begin{array}{l}
p_i^+ (x^*) < -f_2(x^*) \\
p_j^- (x^*) < f_2(x^*) 
\end{array} \Rightarrow p_i^+ (x^*) + p_j^- (x^*) < 0,
$$

\noindent which contradicts (\ref{eq:spliteq}). \qed

\end{enumerate}

\end{proof}

In extended RCP language, the results of Lemma \ref{lemma:innatefathom} have the following respective interpretations:

\begin{enumerate}[(i)]
\item Any combination of linear inequality and equality constraints that are linearly dependent may be fathomed. This is particularly relevant for bound constraint pairs (i.e., $x_i^L \leq x_i$ and $x_i \leq x_i^U$ for $i=1,...,n$) and means that the fathoming basis $F$ may always be initialized as $F := \left[ \begin{array}{ccc}
{\bf 0}_{n \times (n_g-2n)} & I_{n} & I_{n} 
\end{array} \right]$.

\item The combination of non-adjacent pieces in a piecewise-concave approximation where each piece is only maximal over a single continuous interval may be fathomed (see Fig. \ref{fig:ccvinter} for an illustration). Note that both Algorithm~\ref{algo:pwapp} and the standard piecewise-linear approximation of a strict univariate convex function enforce such an approximation.
\item When a univariate nonlinear equality constraint (e.g., ${\rm sin}\;y_1 - z_1 = 0$ in (\ref{eq:decomp2})) is split with both parts strictly underapproximated, the pairs coming from the different approximations cannot define together an optimal active set and may be fathomed. This is equivalent to saying that a strict overapproximation of a function cannot intersect a strict underapproximation of the same function.
\end{enumerate}

\begin{figure*}
\begin{center}
\includegraphics[width=.8\textwidth]{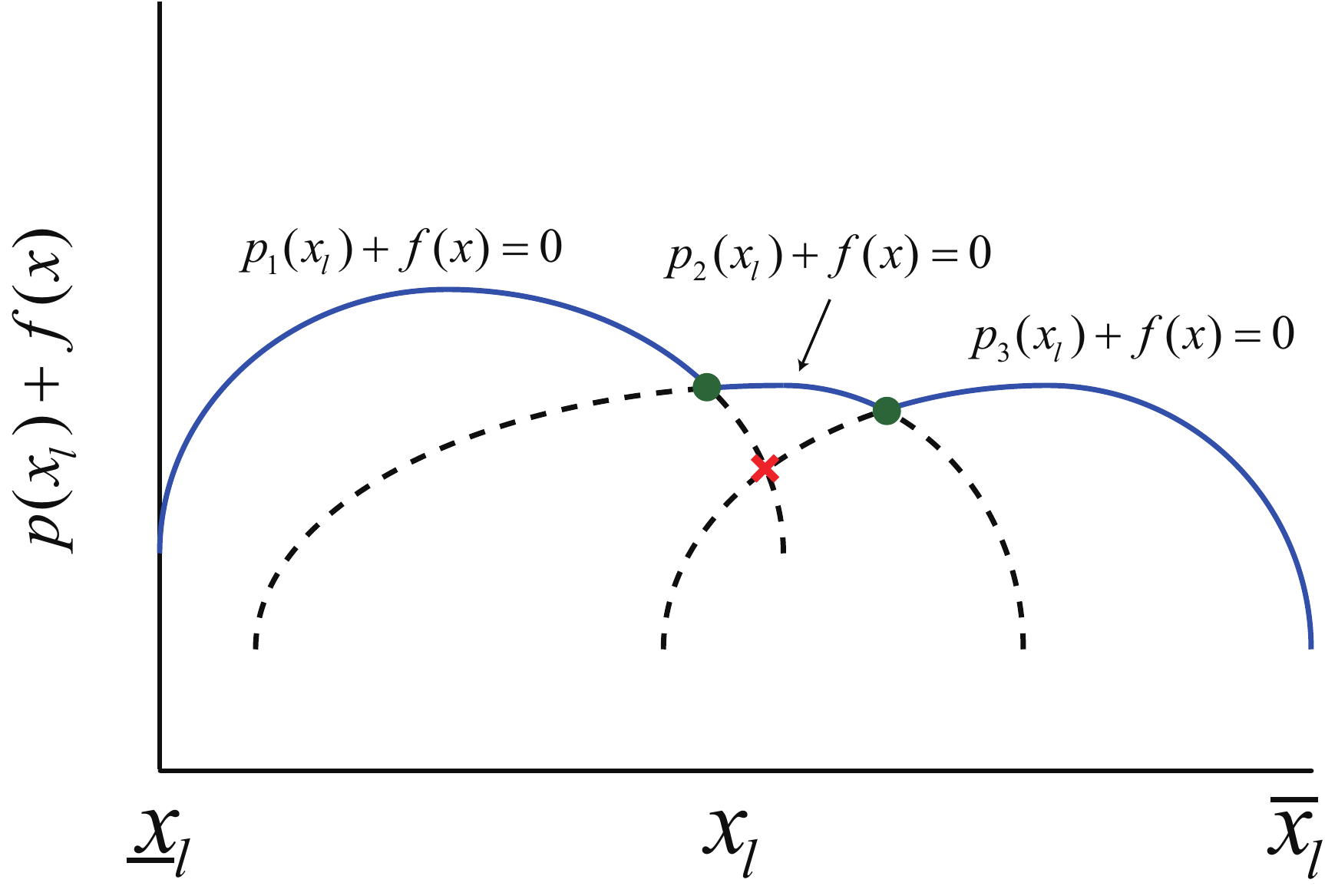}
\caption{Illustration of how a crossing between non-adjacent concave constraints of a piecewise-concave function is bound to lie in the infeasible region when each piece is maximal only over a single interval. Here, the cross designates the intersection between $p_1(x_l)+f(x) = 0$ and $p_3(x_l)+f(x) = 0$, seen to be infeasible, while the two round points indicate the feasible intersections of adjacent constraints $p_1(x_l)+f(x) = p_2(x_l)+f(x)$ and $p_2(x_l)+f(x) = p_3(x_l)+f(x)$.}
\label{fig:ccvinter}
\end{center}
\end{figure*}

\vspace{1mm}
\noindent {\it Fathoming Bound Constraint Combinations}
\vspace{1mm}

A simple way to ensure that Assumption B3 is satisfied is to set lower and upper limits -- $x^L$ and $x^U$, respectively -- on all of the decision variables in the RCP problem. As demonstrated in the proof of Theorem \ref{thm:mainapprox}, such limits exist implicitly for the auxiliary variables $z_{all}$ when Assumptions A1 and A2 hold. The auxiliary epigraph variable $t$ may also be bounded as $\mathop {\min} \{ f_{NL} (y) : y \in \mathcal{Y} \} \leq t \leq \mathop {\max} \{ f_{NL} (y) : y \in \mathcal{Y} \}$. For the sake of simplicity, the algorithm presented in this work \emph{requires} these constraints for all of the variables, not only because this validates Assumption B3, but also because it simplifies the presentation of certain subroutines, which depend on $\mathcal{X}$ being bounded.

The $2n$ bound constraints should, of course, be included in the enumerative search and the active subsets that include their possible combinations should be considered. However, all points corresponding to a given active subset of bound constraints may be proven infeasible, and the constraint combination thereby fathomed, via a cheap computational procedure when $\mathcal{C}$ is a convex polytope defined by a set of linear inequalities

\vspace{-2mm}
$$
\mathcal{C} = \{ x : A_{\mathcal{C},i.} x \leq b_{\mathcal{C},i},\; i = 1,...,n_\mathcal{C} \}.
$$

\noindent The methods proposed in this work will deal exclusively with sets $\mathcal{C}$ having this form.

This procedure is possible since choosing a subset of bound constraints fixes the corresponding variables and allows for a simple minimization of any linear constraint over the rest. Without loss of generality, let $x_1,...,x_{\tilde {\tilde n}}$ denote the $\tilde {\tilde n}$ variables whose bound constraints have been fixed and let $x_{\tilde {\tilde n} + 1},...,x_{n}$ denote the others. Defining 

\vspace{-2mm}
$$
\mathcal{D} = \{ x : x_i^L \leq x_i \leq x_i^U, \; i = 1,...,n \},
$$

\noindent suppose, for the purpose of illustration, that one wants to check if the subset $i_{\tilde A}$ with the first $\tilde {\tilde n}$ lower bound constraints active can be fathomed. To do this, one may minimize the linear portion of each constraint in $\mathcal{C}$ as follows:

\vspace{-2mm}
\begin{equation}\label{eq:minlin}
\begin{array}{rl}
\mathop {\rm minimize}\limits_{x \in \mathcal{D}} & A_{\mathcal{C},i.} x \\
{\rm{subject}}\;{\rm{to}} & x_j = x_j^L, \;\; \forall j = 1,...,{\tilde {\tilde n}},
\end{array}
\end{equation}

\noindent and then check if the objective value is strictly superior to $b_{\mathcal{C},i}$. If so, then it follows that any point in $\mathcal{D}$ with the first $\tilde {\tilde n}$ lower bound constraints active will fail to satisfy $A_{\mathcal{C},i.} x \leq b_{\mathcal{C},i}$, thereby implying that these bound constraints could not define $x^*$ since $x^* \in \mathcal{C}$. This check is carried out for all $i = 1,...,n_{\mathcal{C}}$, with the failure to satisfy any one constraint sufficient to fathom the bound constraint combination.

It is easily shown that the minimal objective value of (\ref{eq:minlin}) has the analytical expression

\vspace{-2mm}
$$
\sum_{j=1}^{\tilde {\tilde n}} A_{\mathcal{C},ij} x_j^L + \sum_{j=\tilde {\tilde n} + 1}^{n} \mathop {\min} \left[ A_{\mathcal{C},ij} x_j^L,\; A_{\mathcal{C},ij} x_j^U \right].
$$

Although this method can provide useful fathoming information at a very low price, running this check for all possible active sets and subsets generated by the bound constraints may nevertheless be computationally expensive. A scheme is therefore proposed to check the different sets in a branching manner that is similar in nature to the general active-set enumeration. As in the general enumeration, subsets from lowest to highest cardinality are built, with each node split in two by activating both the lower and upper bounds of an additional variable and checking if the resulting subsets may be proven infeasible. To ensure that the algorithm terminates fairly quickly, early termination is enforced if the number of nodes grows too large. While this heuristic rule does not rigorously guarantee termination within a certain number of operations, it has been noted to be sufficient in practice, since the number of nodes usually either explodes or stays at reasonable levels, with the algorithm tending to terminate quickly in the latter case.

\vspace{1mm}
\noindent {\bf Subroutine A (Fathoming bound constraints)}
\vspace{1mm}

\noindent {\bf User input}: $\mathcal{C}$, $\mathcal{D}$, $F$, and $M$, with $M$ being the upper limit on the number of nodes.

\noindent {\bf Output}: $F$ (updated).

\begin{enumerate}
\item (Initialize tree) Set $ B := {\bf 0}_{1 \times 2n}$.
\item (Check if search has exhausted all nodes) If $B = \varnothing$, terminate. Otherwise, go to Step 3.
\item (Cardinality and definition of active bound constraints for first node) Set $[\underline b_c^1 \; \overline b_c^1] := B_{1.}$, with $\underline b_c^1,\overline b_c^1 \in \mathbb{R}^{1 \times n}$ containing the first and last $n$ elements of $B_{1.}$, respectively. Remove the first row of $B$ and set $b_c := \underline b_c^1 + \overline b_c^1$. Let $\tilde {\tilde n}$ denote the index of the last non-zero element of $b_c$, with $\tilde {\tilde n} := 0$ if $b_c = {\bf 0}$. If $\tilde {\tilde n} = n$, return to Step 2. Otherwise, proceed to Step 4.
\item (Branching and fathoming) For $k := \tilde {\tilde n} + 1,...,n$:
	\begin{enumerate}
	\item Define $\underline b_c$ and $\overline b_c$ as the vectors $\underline b_c^1$ and $\overline b_c^1$ with their $k^{\rm th}$ indices set to 1.
	\item If

\vspace{-2mm}
\begin{equation}\label{eq:extra1}
\begin{array}{l}
\exists i \in \{ 1,...,n_\mathcal{C} \} : \displaystyle \sum_{j : \underline b_{c,j} = 1} A_{\mathcal{C},ij} x_j^L + \sum_{j : \overline b_{c,j}^1=1} A_{\mathcal{C},ij} x_j^U \vspace{1mm} \\
\hspace{25mm}+ \displaystyle \sum_{j : \underline b_{c,j}, \overline b_{c,j}^1 = 0} \mathop {\min} \left[ A_{\mathcal{C},ij}  x_j^L, \; A_{\mathcal{C},ij} x_j^U \right] > b_{\mathcal{C},i}
\end{array}
\end{equation}

	and 

\begin{equation}\label{eq:extra2}
\not \exists i : F_{ i.} = [{\bf 0}_{1 \times (n_g - 2n)} \;\; \underline b_{c} \;\; \overline b_{c}^1],
\end{equation}

\noindent then the active subset corresponding to $\underline b_{c}$ and $\overline b_{c}^1$ cannot define $x^*$ and may be appended to the fathoming basis:

\vspace{-2mm}
$$
F := \left[  \begin{array}{c}  F \\  {\bf 0}_{1 \times (n_g - 2n)} \;\; \underline b_{c} \;\; \overline b_{c}^1  \end{array}   \right].
$$

\noindent If neither (\ref{eq:extra1}) nor (\ref{eq:extra2}) is true, then this branch should be grown and explored further: 

\vspace{-2mm}
$$
B := \left[   \begin{array}{c} B \\ \underline b_{c} \;\; \overline b_{c}^1  \end{array}   \right].
$$

Likewise, if

\vspace{-2mm}
\begin{equation}\label{eq:extra3}
\begin{array}{l}
\exists i \in \{ 1,...,n_\mathcal{C} \} : \displaystyle \sum_{j : \underline b_{c,j}^1 = 1} A_{\mathcal{C},ij} x_j^L + \sum_{j : \overline b_{c,j}=1} A_{\mathcal{C},ij} x_j^U \vspace{1mm} \\
\hspace{25mm}+ \displaystyle \sum_{j : \underline b_{c,j}^1, \overline b_{c,j} = 0} \mathop {\min} \left[ A_{\mathcal{C},ij} x_j^L,\; A_{\mathcal{C},ij} x_j^U \right] > b_{\mathcal{C},i}
\end{array}
\end{equation}

\noindent and 

\begin{equation}\label{eq:extra4}
\not \exists i : F_{i.} = [{\bf 0}_{1 \times (n_g - 2n)} \;\; \underline b_{c}^1 \;\; \overline b_{c}],
\end{equation}

\noindent then the active subset corresponding to $\underline b_{c}^1$ and $\overline b_{c}$ cannot define $x^*$ and may be added to the fathoming basis:

\vspace{-2mm}
$$
F := \left[  \begin{array}{c}  F \\  {\bf 0}_{1 \times (n_g - 2n)} \;\;\underline b_{c}^1 \;\; \overline b_{c} \end{array}   \right].
$$

\noindent If neither (\ref{eq:extra3}) nor (\ref{eq:extra4}) is true, then this branch should be grown and explored further:

\vspace{-2mm}
$$
B := \left[   \begin{array}{c} B \\ \underline b_{c}^1 \;\; \overline b_{c}  \end{array}   \right].
$$

	\end{enumerate}
\item (Heuristic termination rule if too many nodes) If the number of rows in $B$ is superior to $M$, terminate. Otherwise, return to Step 2. 
\end{enumerate}

\vspace{1mm}
\noindent {\it Fathoming Separable Concave Constraints}
\vspace{1mm}

Let

\vspace{-2mm}
$$
\mathcal{D}_{\mathcal{C}} = \{ x : x_{\mathcal{C},i}^L \leq x_{i} \leq x_{\mathcal{C},i}^U, \;\; i = 1,...,n \}
$$

\noindent denote the box defined by the $\mathcal{C}$-induced bounds $x_{\mathcal{C},i}^L = \mathop {\min} \{ x_i : x \in \mathcal{C} \}$ and $x_{\mathcal{C},i}^U = \mathop {\max} \{ x_i : x \in \mathcal{C} \}$. Clearly, $\mathcal{C} \subseteq \mathcal{D}_{\mathcal{C}}$.

One may prove the inactivity of a given concave constraint over $\mathcal{D}_{\mathcal{C}}$, and thereby $\mathcal{C}$, by computing its maximum value on $\mathcal{D}_{\mathcal{C}}$ and showing that it is strictly inferior to 0. This may be done in the general nonseparable case by solving a single convex optimization problem, and indeed this is what happens in (\ref{eq:feascheck}) when $\# i_{\tilde A} = 1$. However, a faster check may be performed for the separable case since, for a given $g_i$,

\vspace{-2mm}
\begin{equation}\label{eq:maxccv}
\mathop {\max} \limits_{x \in {\mathcal{C}}}\; g_i(x) \leq \mathop {\max} \limits_{x \in \mathcal{D}_{\mathcal{C}}}\; g_i(x) = g_{i0} + \sum_{j=1}^n \mathop {\max} \limits_{x_j \in [x_{\mathcal{C},j}^L, x_{\mathcal{C},j}^U]} g_{ij}(x_j),
\end{equation}

\noindent with $g_{ij}$ denoting the univariate components and $g_{i0}$ in particular denoting the constant term. Since each component is a univariate concave function on a closed interval, each $g_{ij}$ must reach its maximum at either $x_{\mathcal{C},j}^L$, $x_{\mathcal{C},j}^U$, or a stationary point where $dg_{ij}/dx_j = 0$. As checking these cases $n$ times is significantly cheaper than maximizing $g_i$ over $\mathcal{C}$, (\ref{eq:maxccv}) offers an easy way to quickly fathom $g_i$ if it is irrelevant.

\vspace{1mm}
\noindent {\bf Subroutine B (Fathoming separable concave constraints)}
\vspace{1mm}

\noindent {\bf User input}: $\mathcal{D}_{\mathcal{C}}$, $g$, $F$.

\noindent {\bf Output}: $F$ (updated).

\;

\noindent For $i = \{ 1,...,n_g-2n \} \setminus \{ i : g_i\; {\rm not\; separable} \}$:
\begin{enumerate}
\item (Maximize univariate components) For $j = 1,...,n$, compute

\vspace{-2mm}
$$
\begin{array}{l}
\mathop {\max} \limits_{x_j \in [x_{\mathcal{C},j}^L, x_{\mathcal{C},j}^U]} g_{ij}(x_j) = \mathop {\max} \left[ g_{ij} (x_{\mathcal{C},j}^L), g_{ij} (x_{\mathcal{C},j}^U), g_{ij} (x_j^0) \right] \vspace{1mm} \\
\displaystyle x_j^0 \in \left\{ x_j : \frac{dg_{ij}}{dx_j} \Big |_{x_j} = 0, \; x_{\mathcal{C},j}^L \leq x_j \leq x_{\mathcal{C},j}^U \right\},
\end{array}
$$

\noindent with $g_{ij} (x_j^0) := -\infty$ if no $x_j^0$ satisfying the stationarity condition and $x_{\mathcal{C},j}^L \leq x_j^0 \leq x_{\mathcal{C},j}^U$ exists.

\item (Check potential constraint activity over $\mathcal{D}_{\mathcal{C}}$) If

\vspace{-2mm}
$$
g_{i0} + \sum_{j=1}^n \mathop {\max} \limits_{x_j \in [x_{\mathcal{C},j}^L, x_{\mathcal{C},j}^U]} g_{ij}(x_j) < 0,
$$

\noindent and $\not \exists j : F_{j.} = [{\bf 0}_{1 \times (i-1)} \;\; 1 \;\; {\bf 0}_{1 \times ({n_g-2n-i})} \;\;$ $ {\bf 0}_{1 \times 2n} ]$, then set

\vspace{-2mm}
$$
F := \left[ \begin{array}{c}
F \\ {\bf 0}_{1 \times (i-1)} \;\; 1 \;\; {\bf 0}_{1 \times ({n_g-2n-i})} \;\; {\bf 0}_{1 \times 2n}
\end{array} \right].
$$

\end{enumerate}

\vspace{1mm}
\noindent {\it Mandatory Constraints}
\vspace{1mm}

Certain constraints are known to be active at $x^*$ by inspection, and should thus be included in every $i_{\tilde A}$ and $i_A$ considered (i.e., they are ``mandatory''). A very common occurrence of such constraints arises during the epigraph transformation of the cost as shown in (\ref{eq:mainprobeq}), where the added constraint must be active at an optimum, with at least one of the constraints obtained following the decomposition/approximation of this constraint also forced to be active. This may be seen as an implicit fathoming technique since it removes from consideration those sets that do not contain the mandatory constraints without attempting to list all such sets explicitly in $F$.

\subsubsection{General Domain Reduction Techniques}

Domain reduction techniques \cite{Caprara:10} are standard in several of the currently available global optimization solvers \cite{Neumaier:05}, and have been credited for reducing the computational effort of a complete global search significantly \cite{Ryoo:96,Zamora:99}. This is no different in extended RCP, where the use of domain reduction techniques to shrink $\mathcal{D}_{\mathcal{C}}$ and, consequently, $\mathcal{C}$ can make solution times orders of magnitude faster. The particular characteristic of the RCP enumeration with respect to domain reduction techniques is that domain reduction allows for more subsets to be fathomed earlier in the search.

A fairly general domain reduction routine, with a slight modification, is presented first.

\vspace{1mm}
\noindent {\bf Subroutine C (Domain reduction)}
\vspace{1mm}

\noindent {\bf User input}: $\mathcal{D}_{\mathcal{C}}$, $\mathcal{C}$, $X_\mathcal{C}$, and $\epsilon_X$. The matrix $X_\mathcal{C} \in \mathbb{R}^{n \times 2n}$ contains the $2n$ coordinates corresponding to the points where the $n$ variables reach their lower and upper bounds on $\mathcal{C}$. The tolerance $\epsilon_X > 0$ is used to define the termination criterion.

\noindent {\bf Output}: $\mathcal{D}_{\mathcal{C}}$ (updated), $\mathcal{C}$ (updated), $X_\mathcal{C}$ (updated).

\begin{enumerate}
\item (Compute initial coordinates corresponding to $x_{\mathcal{C}}^L, x_{\mathcal{C}}^U$) If $X_\mathcal{C} \neq \varnothing$, go to Step 3. Otherwise, for each $i = 1,...,n$, compute the point corresponding to the minimum value that $x_i$ can take on $\mathcal{C}$:

\vspace{-2mm}
\begin{equation}\label{eq:minx}
\begin{array}{rl}
\underline x_\mathcal{C} \in {\rm arg} \mathop {\rm minimize}\limits_{x} & x_i \\
{\rm{subject}}\;{\rm{to}} & x \in \mathcal{C}
\end{array}
\end{equation}

\noindent and set $x_{\mathcal{C},i}^L := \underline x_{\mathcal{C},i}$. Likewise, compute the point corresponding to the maximum value that $x_i$ can take on $\mathcal{C}$:

\vspace{-2mm}
\begin{equation}\label{eq:maxx}
\begin{array}{rl}
\overline x_\mathcal{C} \in {\rm arg} \mathop {\rm maximize}\limits_{x} & x_i \\
{\rm{subject}}\;{\rm{to}} & x \in \mathcal{C}
\end{array}
\end{equation}

\noindent and set $x_{\mathcal{C},i}^U := \overline x_{\mathcal{C},i}$. Set $X_{\mathcal{C},i.} := [\underline x_\mathcal{C}^T \;\; \overline x_\mathcal{C}^T]$.

\item (Update $\mathcal{C}$) Use the updated $\mathcal{D}_{\mathcal{C}}$ to update $\mathcal{C}$ accordingly. Set $X_{\mathcal{C}}^0 := X_{\mathcal{C}}$.
\item (Re-compute coordinates corresponding to $x_{\mathcal{C}}^L, x_{\mathcal{C}}^U$) For $i = 1,...,n$:
\begin{enumerate}
\item Set $[\underline x_\mathcal{C}^T \;\; \overline x_\mathcal{C}^T ] := X_{\mathcal{C},i.}$, with $\underline x_\mathcal{C}^T, \overline x_\mathcal{C}^T \in \mathbb{R}^n$ vectors corresponding to the first and last $n$ columns of $X_{\mathcal{C},i.}$, respectively.
\item If $\underline x_\mathcal{C} \in \mathcal{C}$, proceed to Step 3c. Otherwise, reorder the constraints of $\mathcal{C}$ so that

\vspace{-2mm}
$$
A_{\mathcal{C},1.} \underline x_\mathcal{C} - b_{\mathcal{C},1} \geq A_{\mathcal{C},2.} \underline x_\mathcal{C} - b_{\mathcal{C},2} \geq ... \geq A_{\mathcal{C},n_\mathcal{C}.} \underline x_\mathcal{C} - b_{\mathcal{C},n_\mathcal{C}}.
$$

\begin{enumerate}[(i)]
\item Set $\tilde n_\mathcal{C} := n - n_C$.
\item Construct

\vspace{-2mm}
\begin{equation}\label{eq:Ab}
A_{\mathcal{C},e} : = \left[ {\begin{array}{*{20}c}
   C \\
   A_{\mathcal{C},1.}   \\
    \vdots   \\
   A_{\mathcal{C},\tilde n_\mathcal{C}.}  \\
\end{array}} \right],\;\;\;b_{\mathcal{C},e} : = \left[ {\begin{array}{*{20}c}
    d \\
    b_{\mathcal{C},1}  \\
    \vdots   \\
   b_{\mathcal{C},\tilde n_\mathcal{C} }  \\
\end{array}} \right].
\end{equation}

\item If ${\rm rank}\; A_{\mathcal{C},e} < n$, set $\tilde n_\mathcal{C} := \tilde n_\mathcal{C} + 1$ and return to (ii). Otherwise, set $\underline x_\mathcal{C} := A_{\mathcal{C},e}^{\dagger} b_{\mathcal{C},e}$.
\item If $\underline x_\mathcal{C} \in \mathcal{C}$, check that $\underline x_\mathcal{C}$ is an optimal solution of (\ref{eq:minx}) by verifying the stationarity condition $\exists \lambda \in \mathbb{R}^{\tilde n_\mathcal{C}}_+, \mu \in \mathbb{R}^{n_{C}} : [{\bf 0}_{1 \times (i-1)} \;\; 1 \;\; {\bf 0}_{1 \times (n-i)} ] + \sum_{i=1}^{\tilde n_\mathcal{C}} \lambda_i A_{\mathcal{C},i.} + \sum_{i=1}^{n_{C}} \mu_i C_{i.} = {\bf 0}$. A cheap way to do this is by taking the pseudoinverse to solve for the Lagrange multipliers. If this cannot be verified, or if $\underline x_\mathcal{C} \not \in \mathcal{C}$, re-compute $\underline x_\mathcal{C}$ by solving (\ref{eq:minx}). Set $x_{\mathcal{C},i}^L := \underline x_{\mathcal{C},i}$.
\end{enumerate}
\item If $\overline x_\mathcal{C} \in \mathcal{C}$, proceed to Step 3d. Otherwise, reorder the constraints of $\mathcal{C}$ so that

\vspace{-2mm}
$$
A_{\mathcal{C},1.} \overline x_\mathcal{C} - b_{\mathcal{C},1} \geq A_{\mathcal{C},2.} \overline x_\mathcal{C} - b_{\mathcal{C},2} \geq ... \geq A_{\mathcal{C},{n_\mathcal{C}.}} \overline x_\mathcal{C} - b_{\mathcal{C},{n_\mathcal{C}}}.
$$

\begin{enumerate}[(i)]
\item Set $\tilde n_\mathcal{C} := n - n_C$.
\item Construct $A_{\mathcal{C},e}$ and $b_{\mathcal{C},e}$ as in (\ref{eq:Ab}).
\item If ${\rm rank}\; A_{\mathcal{C},e} < n$, set $\tilde n_\mathcal{C} := \tilde n_\mathcal{C} + 1$ and return to (ii). Otherwise, define $\overline x_\mathcal{C} := A_{\mathcal{C},e}^{\dagger} b_{\mathcal{C},e}$.
\item If $\overline x_\mathcal{C} \in \mathcal{C}$, check that $\overline x_\mathcal{C}$ is an optimal solution of (\ref{eq:maxx}) by verifying the stationarity condition $\exists \lambda \in \mathbb{R}^{\tilde n_\mathcal{C}}_+, \mu \in \mathbb{R}^{n_C} : [{\bf 0}_{1 \times (i-1)} \;\; -1 \;\; {\bf 0}_{1 \times (n-i)} ] + \sum_{i=1}^{\tilde n_\mathcal{C}} \lambda_i A_{\mathcal{C},i.} + \sum_{i=1}^{n_C} \mu_i C_{i.} = {\bf 0}$. If this cannot be verified, or if $\overline x_\mathcal{C} \not \in \mathcal{C}$, re-compute $\overline x_\mathcal{C}$ by solving (\ref{eq:maxx}). Set $x_{\mathcal{C},i}^U := \overline x_{\mathcal{C},i}$.
\end{enumerate}
\item Set $X_{\mathcal{C},i.} := [\underline x_\mathcal{C}^T \;\; \overline x_\mathcal{C}^T]$.
\end{enumerate}
\item (Termination) If $\| X_{\mathcal{C}} - X_{\mathcal{C}}^0 \|_{\rm max} < \epsilon_X$, terminate. Otherwise, return to Step 2.
\end{enumerate}

The above subroutine essentially updates the box $\mathcal{D}_{\mathcal{C}}$ by solving linear programming (LP) problems to compute the minimal and maximal bounds on the individual variables, which are then used to redefine and shrink $\mathcal{C}$. The aforementioned ``slight modification'' comes via storing the old solution points in $X_{\mathcal{C}}$ and using them whenever possible to avoid solving (\ref{eq:minx}) and (\ref{eq:maxx}), either by (a) verifying that the old point is still inside $\mathcal{C}$ and thus does not need updating or by (b) projecting the old point on the ``most active'' constraints in hopes of this being the active set that would solve (\ref{eq:minx}) or (\ref{eq:maxx}) by simple pseudoinversion. It should be mentioned that storing $X_\mathcal{C}$ may reduce the computational burden of (\ref{eq:minx}) and (\ref{eq:maxx}) significantly as well, as it provides a warm start for what are already LP problems. Other techniques may provide further computational speed-ups (see, e.g., \cite{Zamora:99,Tawarmalani:04}), but have not been considered in this work.

Of crucial importance in Subroutine C is the ``update $\mathcal{C}$ accordingly'' in Step 2, which is pertinent since the definition of $\mathcal{C}$ will in general be dependent on the definition of $\mathcal{D}_{\mathcal{C}}$. Noting that $Cx = d$, any linear $g_i$, and the box $\mathcal{D}$ may be incorporated into $\mathcal{C}$ directly, the other elements that may contribute to the definition of $\mathcal{C}$ -- namely, convex underestimators and cutting planes -- are now discussed in some detail.

\vspace{2mm}
\noindent {\it Convex Underestimators}
\vspace{1mm}

Denote by $l_i(x) \leq g_i (x), \forall x \in \mathcal{D}_{\mathcal{C}}$ a linear underestimator of the (nonlinear) concave constraint $g_i(x) \leq 0$ over $\mathcal{D}_{\mathcal{C}}$. It is well known that the efficiency of such underestimators will depend on the degree of nonlinearity of $g_i$ as well as on the size of $\mathcal{D}_{\mathcal{C}}$. Incorporating $l_i(x) \leq 0$ into $\mathcal{C}$ therefore makes the iterative domain reduction as described in Subroutine~C possible, as tightening $\mathcal{D}_{\mathcal{C}}$ makes the constraint $l_i(x) \leq 0$ more restricting, which in turn allows for further tightening of $\mathcal{D}_{\mathcal{C}}$, and so on. Two particular cases -- arguably the two most relevant ones in the extended RCP methodology -- are addressed here.

The first corresponds to the case where $g_i$ is separable, as this allows one to construct a convex (linear) underestimator of $g_i$ by constructing the convex (linear) underestimators of its univariate components, $g_{ij}$. As all of these components are concave, the convex underestimator of each $g_{ij}$ is simply the line segment joining the points $(x_{\mathcal{C},j}^L, g_{ij}(x_{\mathcal{C},j}^L))$ and $(x_{\mathcal{C},j}^U, g_{ij}(x_{\mathcal{C},j}^U))$. Their sum then gives the underestimator of $g_i$~\cite{Falk:69}.

The second case of interest is that of the univariate piecewise-concave function as generated by Algorithm 1. It is not difficult to show that the convex underestimator of such a function is simply the (piecewise-linear) convex underestimator of the intersection points of the adjacent pieces, together with the points corresponding to the lower and upper boundaries $x_{\mathcal{C}}^L$ and $x_{\mathcal{C}}^U$. A number of algorithms designed for the more general problem of computing the convex hull of a planar set are readily applicable to compute the convex underestimator here. The algorithm employed in this work was a modified version of Graham's method \cite{Graham:72}.

\vspace{2mm}
\noindent {\it Local Minimization and Cutting Planes}
\vspace{1mm}

Given some feasible point $x_0$, it is often reasonable to put in the computational effort for a local optimization so as to bring this point to a local minimum, $x_{loc}^*$, of the RCP problem. The resulting point, in some cases already the global minimum, then gives an upper bound on the globally optimal cost value and allows for the cutting plane constraint $c^T x \leq c^T x_{loc}^*$ to be added to $\mathcal{C}$. In the author's experience, this is arguably the most important constraint with respect to the domain reduction, as finding a very good upper bound on the cost tends to lead to drastic reductions in $\mathcal{D}_{\mathcal{C}}$.

\vspace{1mm}
\noindent {\bf Subroutine D (Local minimization)}
\vspace{1mm}

\noindent {\bf User input}: $x_0$ (optional), $x_{up}$ (optional), $U$, $\mathcal{D}_{\mathcal{C}}$, $\mathcal{C}$, $g_i$, $C$, and $d$, where $U$ is an upper bound on the random samples used to find $x_0$ if it is not provided, while $x_{up}$ is the best known point satisfying $c^T x^* \leq c^T x_{up}$.

\noindent {\bf Output}: $\mathcal{C}$ (updated), $x_{up}$ (updated).

\begin{enumerate}
\item (Search for a feasible point) If a feasible $x_0$ is provided, proceed to Step 2. Otherwise, randomly sample $\mathcal{D}_{\mathcal{C}}$ until (a) a feasible $x_0$ is found or (b) $U$ samples have failed to find a feasible point. In the case of (a), proceed to Step 2. Otherwise, terminate.
\item (Local minimization of the RCP problem) Initialize a local solver at $x_0$ and solve (\ref{eq:RCPprob}) to local optimality to obtain $x_{loc}^*$.
\item (Updating the cost cutting plane) If $c^T x_{loc}^* < c^T x_{up}$, set $x_{up} := x_{loc}^*$ and replace the cost cutting plane in $\mathcal{C}$ with $c^T x \leq c^T x_{loc}^*$. If no $x_{up}$ was provided, then simply add this constraint to $\mathcal{C}$.

\end{enumerate}

\subsection{Detailed Outline of the Proposed RCP Algorithm}
\label{sec:algo3}

Bringing together the ideas of the previous two subsections, the entire algorithm is now presented.

\begin{algorithm}[Enumerative RCP algorithm]\label{algo:RCPfull}

\noindent {\bf User input}: $c$, $g_i$, $C$, $d$, $\mathcal{C}$, $\mathcal{D}$, $M$, $\mathcal{D}_{\mathcal{C}}$, $U$, $F$, $\epsilon_X$, $\epsilon_g$, and $\epsilon$. The matrix $F$ should be populated according to the innate fathoming rules described in Section \ref{sec:algo2}. The tolerances $\epsilon_g > 0$ and $\epsilon > 0$ correspond to acceptable constraint violations and suboptimality, respectively. It is assumed that both $Cx = d$ and $x \in \mathcal{D}$ are included in the definition of $\mathcal{C}$.

\noindent {\bf Output}: $X^*$ (the matrix of solution candidates).

\begin{enumerate}
\item (Initialization) Set $N := 0$ as the counter for the optimization problems solved. Set $X^*, V := \varnothing$. Set $X_V := \varnothing$ as the matrix of points corresponding to the different members of $V$. Set $S := [{\bf 1}_{1 \times n_m}\; {\bf 0}_{1 \times (n_g-n_m)}]$ as the binary matrix corresponding to the candidate active sets and subsets -- initially a vector with only the $n_m$ mandatory constraints accounted for. Set $s_{low} := 0$ as the vector of lower bounds corresponding to the constraint sets in $S$, initially a scalar with the dummy value of 0 that serves as a place holder. Set $X_\mathcal{C} := \varnothing$.
\item (Initial upper bound) Run Subroutine D with no $x_0, x_{up}$ provided (unless a feasible point is somehow known, in which case use this point as $x_0$), and augment the counter as $N := N + 1$. If Subroutine D fails to find a feasible point, set $x_{up}$ as

\vspace{-2mm}
$$
\begin{array}{rl}
x_{up} \in {\rm arg} \mathop {\rm maximize}\limits_{x} & c^T x \\
{\rm{subject}}\;{\rm{to}} & x \in \mathcal{C}
\end{array}
$$

\noindent and augment the counter as $N := N + 1$.

\item (Initial domain reduction) Run Subroutine C and augment the counter as $N := N + N_C$, where $N_C$ is the number of times that Problems (\ref{eq:minx}) or (\ref{eq:maxx}) are solved.
\item (Fathoming of bound and separable constraints) Run Subroutines A and B.
\item (Solving the convex relaxation) Solve the relaxed problem

\vspace{-2mm}
\begin{equation}\label{eq:cvxrelax}
\begin{array}{rl}
x_{low} \in {\rm arg} \mathop {\rm minimize}\limits_{x} & c^T x \\
{\rm{subject}}\;{\rm{to}} & x \in \mathcal{C},
\end{array}
\end{equation}

\noindent and augment the counter as $N := N + 1$. If $\mathop {\max} \limits_{i = 1,...,n_g} g_i (x_{low}) \leq \epsilon_g$, then terminate and declare $X^* := x_{low}^T$. Alternatively, if $c^T x_{up} - c^T x_{low} \leq \epsilon$, terminate and declare $X^* := x_{up}^T$.

\item (Setting the node to branch on) If $S = \varnothing$, terminate. Otherwise, set $\tilde s^1 := S_{1.}$. Remove the first row of $S$ and the first element of $s_{low}$. Denote by $\tilde n_g$ the index of the last non-zero element of $\tilde s^1$, setting $\tilde n_g := 0$ if $\tilde s^1 = {\bf 0}$. If $\| \tilde s^1 \|_1 < n-n_C-1$, proceed to Step 7. Otherwise, proceed to Step 8. 
\item (Checking active subsets) Set $i_k := \{ i : \tilde n_g + 1 \leq i \leq n_g \}$, ordered in increasing order, and define those indices of $i_k$ for which individual constraints have been fathomed as $i_F := \{ i : \exists j : F_{j.} = [{\bf 0}_{1 \times (i - 1)} \;\; 1 $  $ {\bf 0}_{1 \times (n_g - i)} ]$, $i \geq \tilde n_g + 1 \}$. Set $i_k := i_k \setminus i_F$. Remove the last $n - n_C - \| \tilde s^1 \|_1 - 1$ elements of $i_k$ to avoid exploring those branches that terminate without being able to reach full cardinality, i.e., $n - n_C$ members. Set $k$ equal to the first element of $i_k$, then:
\begin{enumerate}
\item (Choose subset) Define the candidate subset, $\tilde s$, as $\tilde s^1$ with the $k^{\rm th}$ element set to 1. Define the corresponding index set as $i_{\tilde A} = \{ i : \tilde s_i = 1 \}$.
\item (Check if subset is spanned by fathoming basis) If $\exists i : F_{i.} \subseteq_B \tilde s$, then $i_{\tilde A} \not \subset i_{A^*}$ and may be fathomed. Proceed to Step 7g.
\item (Check if subset belongs to validation basis) If $\exists i : \tilde s \subseteq_B V_{i.}$, then Problem (\ref{eq:feascheck}) must be feasible for $i_{\tilde A}$ and may be skipped. Its corresponding lowest cost value may be set to the best available lower bound on the cost. Expand the tree by setting

\vspace{-2mm}
$$
S := \left[ \begin{array}{c}
S \\ \tilde s
\end{array} \right], \;\;\; s_{low} := \left[ \begin{array}{c}
s_{low} \\ c^T x_{low}
\end{array} \right]
$$

\noindent and proceed to Step 7g.
\item (Check validity of subset by solving (\ref{eq:feascheck})) Solve (\ref{eq:feascheck}) and augment the counter as $N := N + 1$. If infeasible, then $i_{\tilde A} \not \subset i_{A^*}$ and may be fathomed. Since it is not spanned by the fathoming basis already (or it would have been removed in Step 7b), add it to the fathoming basis by setting

\vspace{-2mm}
$$
F := \left[ \begin{array}{c}
F \\ \tilde s
\end{array} \right]
$$

\noindent and proceed to Step 7g. If (\ref{eq:feascheck}) is feasible, expand the tree by setting

\vspace{-2mm}
$$
S := \left[ \begin{array}{c}
S \\ \tilde s
\end{array} \right], \;\;\; s_{low} := \left[ \begin{array}{c}
s_{low} \\ c^T \tilde x^*
\end{array} \right].
$$

\item (Update the validation basis) Set $i_V := \{ i : g_i (\tilde x^*) \geq 0 \}$ as the index set of active constraints at $\tilde x^*$, and let $v_c$ denote its corresponding binary (row) vector, so that $v_{c,i} = 1, \; \forall i \in i_V$. Remove any rows $i$ from $V$ for which $V_{i.} \subseteq_B v_c$, together with the corresponding rows from $X_V$, as these are now redundant. Update $V$ and $X_V$ as

\vspace{-2mm}
$$
V := \left[ \begin{array}{c}
V \\ v_c
\end{array} \right], \;\;\; X_V := \left[ \begin{array}{c}
X_V \\ (\tilde x^*)^T
\end{array} \right].
$$

\item (Local minimization and domain reduction) If $g_i (\tilde x^*) \leq 0, \; \forall i = 1,...,n_g$ or if $N > 50$, run Subroutine D ($N := N+1$) with $\tilde x^*$ as the initial point in the former case and no initial point in the latter (in this case, reset $N := 0$). If $x_{up}$ changes, follow with Subroutines C ($N := N + N_C$), A, and B, and then repeat the procedure of Step 5. For any rows $i$ of $S$ where $s_{low,i} > c^T x_{up}$, transfer the corresponding rows $S_{i.}$ to $F$ and delete these elements from $s_{low}$, as none of these subsets can contribute to defining $x^*$ since the lowest cost value they can achieve is superior to $c^T x^*$. Find any indices $i : (X_{V,i.})^T \not \in \mathcal{C}$ and remove these rows from $V$ and from $X_V$, as these entries of the validation basis are no longer valid for the updated $\mathcal{C}$.
\item (Proceed to next branch) If $k$ is the last element of $i_k$, return to Step 6. Otherwise, set $k$ as equal to the next element of $i_k$ and return to 7a.
\end{enumerate}

\item (Active set enumeration) Set $i_k := \{ i : \tilde n_g + 1 \leq i \leq n_g \}$, define $i_F$ as in Step 7, and set $i_k := i_k \setminus i_F$. Set $k$ equal to the first element of $i_k$, then: 
\begin{enumerate}
\item (Choose active set candidate) Define the candidate active set, $\tilde s$, as $\tilde s^1$ with the $k^{\rm th}$ index set to 1. Define the corresponding index set as $i_{A} = \{ i : \tilde s_i = 1 \}$.
\item (Check if set is spanned by fathoming basis) If $\exists i : F_{i.} \subseteq_B \tilde s$, then $i_A \neq i_{A^*}$ and may be ignored. Proceed to Step 8d.
\item (Compute global optimum candidate) Solve (\ref{eq:RCPprobrevgen}), denoting the solution by $x_{cand}^*$. If $g_i(x_{cand}^*) \leq \epsilon_g, \; \forall i = 1,...,n_g$, then append $x_{cand}^*$ to the solution set

\vspace{-2mm}
$$
X^* := \left[ \begin{array}{c}
X^* \\ (x_{cand}^*)^T
\end{array} \right].
$$

\item (Proceed to next active set candidate) If $k$ is the last element of $i_k$, return to Step 6. Otherwise, set $k$ as equal to the next element of $i_k$ and return to 8a.
\end{enumerate}
\end{enumerate}

\end{algorithm}

Some remarks:

\begin{itemize}
\item There are three ways for Algorithm \ref{algo:RCPfull} to terminate. Criteria I and II will be defined as termination due to a sufficiently tight $\mathcal{C}$, which yields a relaxed solution that either, in the case of I, satisfies the concave constraints with an acceptable tolerance $\epsilon_g$ or, in the case of II, yields a lower bound on the cost that is sufficiently close to the value at a local minimum that has already been found. Criterion III indicates that the full enumeration has been carried out, in which case the full set of candidates $X^*$ is reported -- the member(s) with the lowest cost value corresponding to the global minimum (minima). If $X^*$ is empty, then this implies that the RCP problem is infeasible.
\item Note that Termination Criteria I and II do not require the RCP problem to be regular, as both declare a solution by more traditional means. Regularity is required for Termination Criterion III to be valid, however. Also note that Criterion III will yield \emph{all} global minima in the case that multiple minima exist, while I and II may terminate as soon as just one of these is found and proven to be globally optimal within a certain tolerance.
\item Some care should be taken with respect to the numerical tolerances of the optimization problems and subroutines involved in the algorithm, as failing to do so may lead to a nonrobust implementation with some feasible solutions being fathomed due to slight numerical infeasibility. As just one example, consider the case of a local solver finding the globally minimal cost (in Subroutine D) lowered by a numerical error of $-10^{-4}$, and thereby reporting $c^T x_{up} = c^T x^* - 10^{-4}$. If this is then incorporated into $\mathcal{C}$ as a cutting plane constraint on the cost and is used by a different solver to solve the domain reduction and relaxed problems, it may be that the latter cannot find a feasible solution as the reported upper bound is slightly below what is feasible. Details regarding where all such tolerances should be accounted for would result in a lengthy discussion, which the reader is spared, but it is worth noting that they are quite important nevertheless.
\item The counter $N$ adds a heuristic rule by which the RCP solver decides to ``take a break'' from the enumeration to perform a local minimization and hopefully find a new local optimum with which to refine $\mathcal{C}$. This is the only non-deterministic feature of the algorithm, since the initial starting point for the local minimization will be randomly generated.
\item Algorithm \ref{algo:RCPfull} builds the active-set tree (e.g., Fig. \ref{fig:branchtree}) dimension by dimension, which results in Step 7 being exhausted before Step 8 is reached, with the latter corresponding to the solution of the reverse problems (\ref{eq:RCPprobrevgen}) for any active sets that have not been fathomed. Since the validation basis $V$ is no longer needed in Step 8, it is no longer updated or used there.
\item The choice to order the elements of $i_k$ in increasing order is not mandatory, and other choices could be proposed. Essentially, this affects how the constraints are ordered when growing the branches, and is likely to affect performance. It is difficult to say if this choice could be optimized, although strategies that are analogous to those used in the more standard branch-and-bound schemes \cite{Neumaier:04,Tawarmalani:04,Achterberg:05} could be proposed.
\end{itemize}

\section{Illustrative Examples}
\label{sec:ex}

Five NLP examples are chosen to illustrate the strengths and weaknesses of the proposed algorithm. Of these, the first two are problems that are already in standard RCP form and do not need approximation, the third is a concave minimization problem that is put into the standard form via the epigraph transformation, and the last two are NLP problems that are approximated and solved as RCPs. The size of the problems ranges from $n = 2$ to $n = 200$, although for the most part the problems are small and intended only to illustrate the viability of the proposed method.

The algorithm and all subroutines were coded in MATLAB$\textsuperscript{\textregistered}$, with the CVX-SeDuMi modeling-solver combination \cite{Sturm:99,Grant:08,CVX:12} used for solving all of the convex subproblems. All local (nonconvex) minimizations were done with the MATLAB routine \texttt {fmincon}. For all optimizations, it was verified that the solution converged to a local minimum using the built-in verification mechanisms of each solver. Since the dominant computational effort of the framework lies in the number of optimization problems solved by Algorithm \ref{algo:RCPfull}, the computational effort for each example is reported in terms of the number of times that each type of optimizer is called, with the following three types being relevant:

\begin{itemize}
\item ``Convex'': Problems (\ref{eq:feascheck}) and (\ref{eq:RCPprobrevgen}), which are general convex NLPs. With the exception of Example 2, these are always quadratically constrained problems with a linear cost.
\item ``LP'': Given as two numbers, $N_1 + N_2$, with $N_1$ denoting the LP problems (\ref{eq:minx}) and (\ref{eq:maxx}) solved during domain reduction and $N_2$ denoting the LP relaxation (\ref{eq:cvxrelax}) solved over $\mathcal{C}$.
\item ``Local'': Problem (\ref{eq:RCPprob}) solved to local optimality.
\end{itemize}

\noindent The tolerances for the algorithm were set as $\epsilon := 10^{-3}, \epsilon_g := 10^{-6}, \epsilon_X := 10^{-4}$, with the values $M := 100$ and $U := 10^6$ used for Subroutines A and D, respectively. 

The RCP regularity of each problem is verified by confirming that $x^*$ must be a strict local minimum. Without going into the details of each example, the general procedure that is applicable to all of the problems considered here is summarized as follows:

\begin{enumerate}
\item By contradiction, it is supposed that $x^*$ is not a strict local minimum and that there exists a feasible direction in the null space of $\left[ \begin{array}{c} c^T \\ C \end{array} \right]$.
\item Analyzing all such directions shows that the optimal active set indexed by $i_{A^*}$ cannot remain active in these directions due to the negative definiteness (strict concavity) of some of its elements.
\end{enumerate}

\noindent The interested reader is referred to an earlier draft of this paper for the proofs for each problem individually \cite[\S 5]{Bunin:2013}.

\vspace{2mm}
\noindent {\it Example 2 (A low-dimensional problem with concave constraints)}
\vspace{2mm}

The following problem is solved:

\vspace{-2mm}
\begin{equation}\label{eq:example1}
\begin{array}{rl}
\mathop {\rm minimize}\limits_{x_1, x_2} & c^T x \\
{\rm{subject}}\;{\rm{to}} & -2.42 (x_1 + 0.4)^2 + 1.1 x_1 + x_2 - 0.235 \leq 0 \\
& -1.1x_1^2 + 1.3x_1 - x_2 - 0.17 \leq 0 \\
& -e^{-5x_1 + 4} - x_2 + 1.2 \leq 0 \\
& -(x_1-0.5)^2 - (x_2 - 0.5)^2 + 0.09 \leq 0 \\
& -22(x_1 - 0.3)^2 + 1.1x_1 + x_2 - 1.155 \leq 0 \\
& -2.2(x_1 - 0.5)^2 + 1.1x_1 + x_2 - 1.475 \leq 0 \\
& -20(x_1 - 0.1)^2 + 1.3x_1 - x_2 + 0.5 \leq 0 \\
& -x_i \leq 0, \; x_i -1 \leq 0, \; i = 1,2,
\end{array}
\end{equation}

\noindent which is already in standard RCP form and has, as one of its major characteristics, a disconnected feasible region (Fig. \ref{fig:ex1}).

\begin{figure*}
\begin{center}
\includegraphics[width=0.75\textwidth]{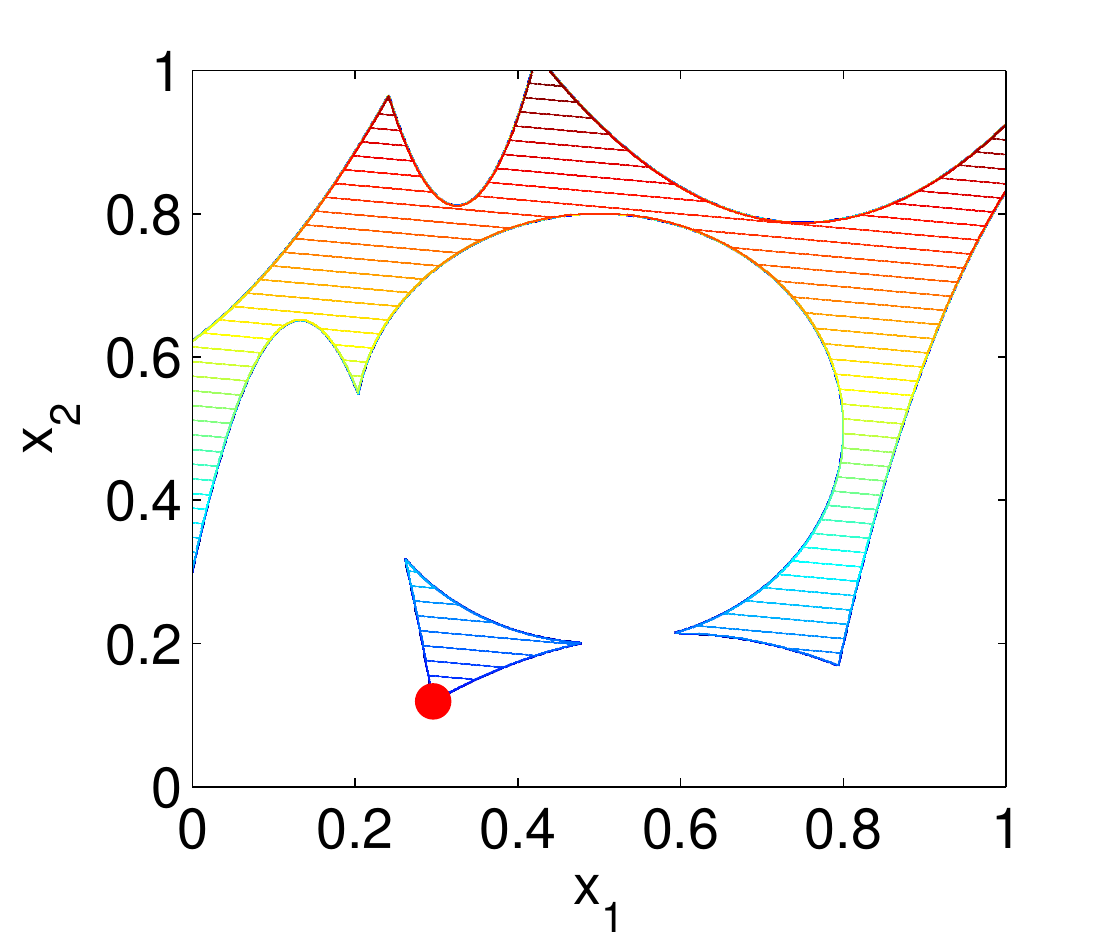}
\caption{The feasible region (lined) of (\ref{eq:example1}) with the cost contours and global minimum given for the case of $c_1 = 0.1, \;c_2 = 1.0$.}
\label{fig:ex1}
\end{center}
\end{figure*}

Computational results for ten randomly generated $c$ are reported in Table \ref{tab:ex1}, and it is seen that the computational burden for this problem is quite light. It is worth noting that the brute enumeration approach would require solving $\binom{11}{2} = 55$ convex problems to arrive at the solution, which, though probably acceptable, still requires significantly more computation than Algorithm \ref{algo:RCPfull}. Finally, one sees that in over half of the cases, domain reduction finds the solution before the enumeration begins.

\begin{table}
\begin{center}
\caption{Computational effort for Example 2.}
\label{tab:ex1}       % Give a unique label
\begin{tabular}{rrllll}
\hline\noalign{\smallskip}
$c_1$ & $c_2$ & Convex & LP & Local & Termination  \\
\noalign{\smallskip}\hline\noalign{\smallskip}
0.1 & 1.0  & 3 & 7 + 2 & 2 & I \\
1.0 & 0.4 & 0 & 4 + 1 & 1 & I \\
$-0.2$ & 0.7 & 0 & 6 + 1 & 1 & I \\
$-0.1$ & 0.1 & 0 & 5 + 1 & 1 & I \\
$-0.6$ & 2.2 & 0 & 8 + 1 & 1 & I \\
0.7 & 1.6 & 4 & 8 + 2 & 2 & I \\
0.6 & $-0.6$ & 13 & 4 + 1 & 1 & III \\
1.1 & 0.1 & 0 & 5 + 1 & 1 & I \\
$-0.1$ & $-0.8$ & 0 & 4 + 1 & 1 & I \\
0.3 & $-1.3$ & 6 & 4 + 1 & 1 & III \\
\noalign{\smallskip}\hline
\end{tabular}
\end{center}
\end{table}

\vspace{2mm}
\noindent {\it Example 3 (High-dimensional RCP problems with favorable complexity)}
\vspace{1mm}

Consider the RCP problem

\vspace{-2mm}
\begin{equation}\label{eq:exRCPsup}
\begin{array}{rl}
\mathop {\rm minimize}\limits_{x} & \displaystyle \sum \limits_{i = 1}^{n} x_i \\
{\rm{subject}}\;{\rm{to}} & 1 - \displaystyle \sum \limits_{i = 1}^{n} w_i x^2_i \leq 0\\
& -x_i \leq 0, \;\; i = 1,...,n,
\end{array}
\end{equation}

\noindent with $w \in \mathbb{R}^{n}_{++}$ a random vector with $\| w \|_\infty \leq 1$. A two-dimensional cut of this problem is shown in Fig. \ref{fig:goodex}, from which it is easily seen that the difficulty arises from the ellipse centered at the origin, generated by the single strictly concave constraint. This is, however, an example of an RCP with favorable complexity, as the number of active sets, without any fathoming, is equal to $\binom{n+1}{n} = n+1$ and scales linearly in $n$. As such, one could always solve this problem by solving $n$ convex optimization problems (the active set corresponding to $x = {\bf 0}$ may be fathomed as the solution for this set is clearly infeasible). In fact, one could do even better and apply RCP theory directly, from which it is known that the global optimum must lie at a point where $n-1$ of the variables are 0, thus leading to the analytical solution

\vspace{-2mm}
$$
x_i^* = \left\{ \begin{array}{rl} \displaystyle \sqrt{\frac{1}{w_i}}, & \;\;\; \displaystyle w_i = \| w \|_\infty \vspace{1mm} \\ 0, & \;\;\; \rm{otherwise} \end{array} \right .
$$

\noindent for the non-pathological case where only one element of $w$ is equal to $\| w \|_\infty$.

\begin{figure*}
\begin{center}
\includegraphics[width=0.5\textwidth]{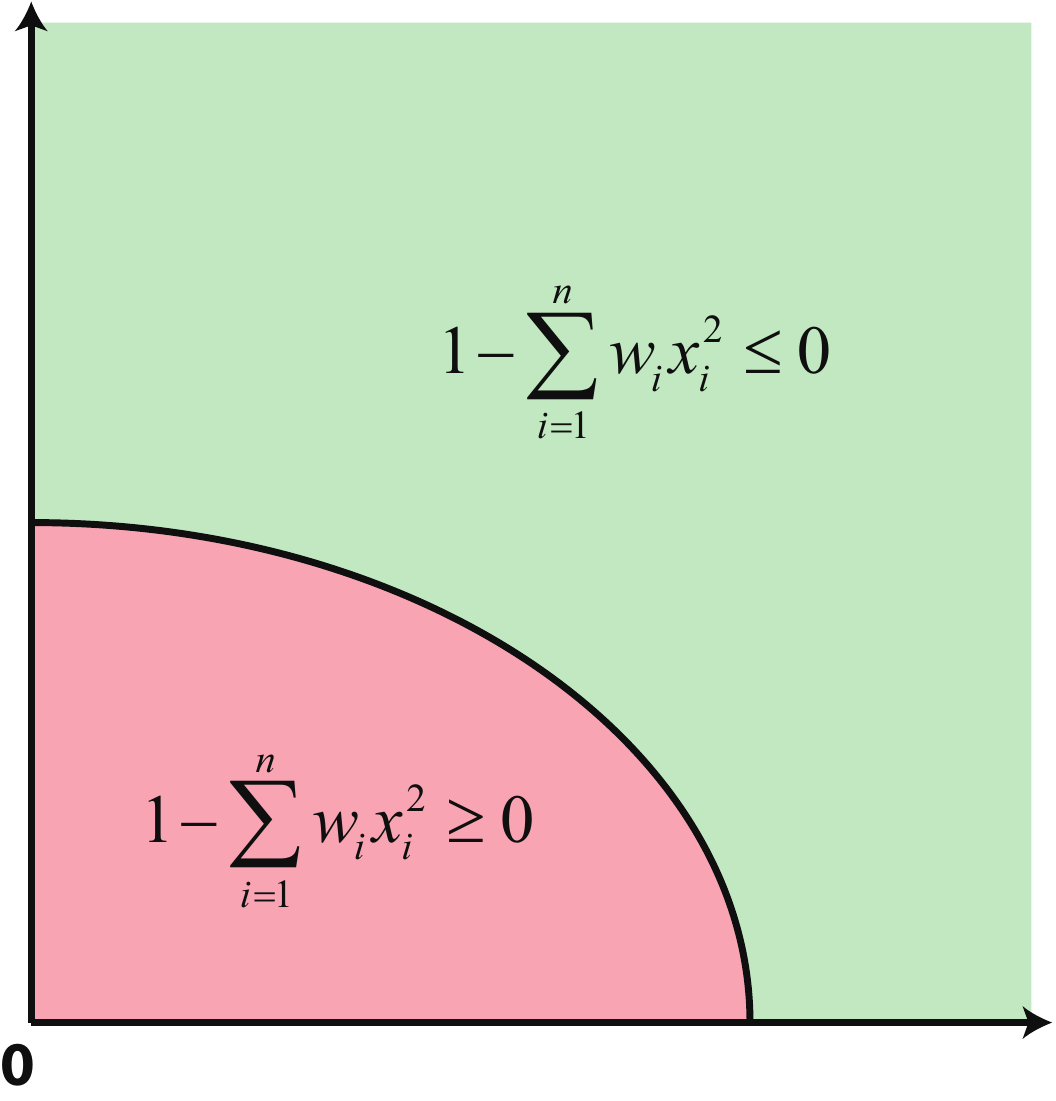}
\caption{A two-dimensional cut of the decision-variable space from the example in (\ref{eq:exRCPsup}), where the strict positivity of the cost vector coefficients ensures that the solution always lie on the intersection of the strictly concave constraint and $n-1$ of the bound constraints.}
\label{fig:goodex}
\end{center}
\end{figure*}

To test how the proposed method solves the problem, Algorithm \ref{algo:RCPfull} is run for various dimension sizes $n$. Because the algorithm requires upper bounds on the variables as well, an additional set of dummy constraints, $x_i - 100 \leq 0, \;i = 1,...,n$, is provided, although these are fathomed  during initialization. Table \ref{tab:ex2} reports the results, where it is seen that the domain reduction techniques alone are able to solve the problem in all of the cases except for $n = 20$ and $n = 200$. For the most part, the number of LP problems that are solved seems to scale well with $n$, with the problems corresponding to $n=120$ and $n = 180$ being notable exceptions. For all of the examples considered, the total number of convex problems solved, excluding the LP and Local problems, is always inferior to the quantity that would be required by the brute enumeration ($n$).

\begin{table}
\begin{center}
\caption{Computational effort for Example 3.}
\label{tab:ex2}       % Give a unique label
\begin{tabular}{lllll}
\hline\noalign{\smallskip}
$n$ & Convex & LP & Local & Termination  \\
\noalign{\smallskip}\hline\noalign{\smallskip}
20 & 20 & 118 + 2 & 2 & III \\
40 & 1 & 299 + 2 & 2 & I \\
60 & 1 & 441 + 2 & 2 & I \\
80 & 0 & 433 + 1 & 1 & I \\
100 & 1 & 589 + 2 & 2 & I \\
120 & 43 & 1834 + 4 & 4 & I \\
140 & 0 & 574 + 1 & 1 & I \\
160 & 0 & 787 + 1 & 1 & I \\
180 & 3 & 3206 + 4 & 4 & I \\
200 & 159 & 2523 + 5 & 5 & III \\
\noalign{\smallskip}\hline
\end{tabular}
\end{center}
\end{table}

\vspace{2mm}
\noindent {\it Example 4 (Concave minimization)}
\vspace{1mm}

The following problem is solved:

\vspace{-2mm}
$$
\begin{array}{rl}
\mathop {\rm minimize}\limits_{y} & -50 y^T y + \alpha c_y^T y \\
{\rm{subject}}\;{\rm{to}} & Ay \leq b \\
& y_i \in [0,1],\;\; i = 1,...,10,
\end{array}
$$

\noindent where $A$ and $b$ are defined as

\vspace{-2mm}
$$
A = \left[ {\begin{array}{*{20}r}
   2 & { - 6} & { - 1} & 0 & { - 3} & { - 3} & { - 2} & { - 6} & { - 2} & { - 2}  \\
   6 & { - 5} & 8 & { - 3} & 0 & 1 & 3 & 8 & 9 & { - 3}  \\
   { - 5} & 6 & 5 & 3 & 8 & { - 8} & 9 & 2 & 0 & { - 9}  \\
   9 & 5 & 0 & { - 9} & 1 & { - 8} & 3 & { - 9} & { - 9} & { - 3}  \\
   { - 8} & 7 & { - 4} & { - 5} & { - 9} & 1 & { - 7} & { - 1} & 3 & { - 2}  \\
\end{array}} \right],\;\;b = \left[ {\begin{array}{*{20}r}
   { - 4}  \\
   {22}  \\
   { - 6}  \\
   { - 23}  \\
   { - 12}  \\
\end{array}} \right],
$$

\noindent and where $c_y = [48 \;\; 42\;\;48\;\;45\;\;44\;\;41\;\;47\;\;42\;\;45\;\;46]^T$. The scalar $\alpha$ is varied for test purposes, with $\alpha := 1$ corresponding to Test Problem 2.6 from \cite{Floudas:99}.

This problem is easily converted into standard RCP form by applying the epigraph transformation to the nonlinear portion of the cost function:

\vspace{-2mm}
$$
\begin{array}{rl}
\mathop {\rm minimize}\limits_{x} & \left[  \begin{array}{c} \alpha c_y \\ 1   \end{array}  \right]^T  x \\
{\rm{subject}}\;{\rm{to}} & -50x^T \left[ \begin{array}{cc} I_{10} & \\ & 0  \end{array}  \right] x - x_{11} \leq 0 \\
 & [ A \; {\bf 0}_{5 \times 1} ] x - b \leq 0 \\
 & -x_i \leq 0, \; x_i - 1 \leq 0, \; i = 1,...,10 \\
 & -x_{11} -500 \leq 0, \; x_{11} \leq 0.
\end{array}
$$

\noindent The bounds on the auxiliary variable $x_{11}$ correspond to a conservative approximation of the minimum and maximum values of $-50y^T y$ over $\mathcal{Y}$. It is clear that the new constraint is mandatory and must belong to any active set that is considered.

Table \ref{tab:ex3} presents the computational results for different values of $\alpha$. With the exception of the problems corresponding to $\alpha := 0.1$ and $\alpha := 1$, it is seen that domain reduction is once again very effective here, with a single initial local minimization sufficient to allow for the scheme to reduce the domain and find $x^*$ by solving the relaxed problem. For the two problems where this does not occur and where more computation is needed, it is not entirely inappropriate to blame the issues on ``bad luck'' -- the subroutines are simply unable to find useful local minima with which to reduce the domain. This is particularly clear for $\alpha := 1$, where one sees that the majority of the effort goes into solving the convex problems and not reducing the domain (in a ratio of $2413 : 50$), in spite of a fairly large number of local minimizations performed (i.e., 28). As a result, the algorithm goes through the full enumeration and terminates by Criterion III. While solving 2413 convex problems is probably not desirable, the algorithm is nevertheless orders of magnitude more efficient than the brute enumeration, which would require solving $\binom{27}{10} = 8436285$ convex problems if it is taken into account that the first constraint is mandatory.

\begin{table}
\begin{center}
\caption{Computational effort for Example 4.}
\label{tab:ex3}       % Give a unique label
\begin{tabular}{rllll}
\hline\noalign{\smallskip}
$\alpha$ & Convex & LP & Local & Termination  \\
\noalign{\smallskip}\hline\noalign{\smallskip}
$-10$ & 0 & 63 + 1 & 1 & I \\
$-1$ & 0 & 73 + 1 & 1 & I \\
$-0.1$ & 0 & 66 + 1 & 1 & I \\
$0$ & 0 & 68 + 1 & 1 & I \\
$0.1$ & 252 & 142 + 3 & 6 & I \\
$1$ & 2413 & 50 + 12 & 28 & III \\
$10$ & 0 & 49 + 1 & 1 & I \\
\noalign{\smallskip}\hline
\end{tabular}
\end{center}
\end{table}

\vspace{3mm}
\noindent {\it Example 5 (Extended RCP with lower and upper bounds on the optimal cost)}
\vspace{1mm}

The problem of Al-Khayyal and Falk \cite{AlKhayyal:83} is approximated by an RCP problem:

$$
\begin{array}{rl}
\mathop {\rm minimize}\limits_{y} & -y_1 + y_1 y_2 - y_2 \\
{\rm{subject}}\;{\rm{to}} & -6 y_1 + 8 y_2 \leq 3 \\
& 3 y_1 - y_2 \leq 3 \\
& y_1, y_2 \in [0, 5]
\end{array} \;\;\; \approx \;\;\;
\begin{array}{rl}
\mathop {\rm minimize}\limits_{x} & -x_1 - x_2 + 0.5 x_3 \\
{\rm{subject}}\;{\rm{to}} & p_i (x_4) - x_1^2 - x_2^2 - x_3 \leq 0, \;\;i = 1,...,n_p \\
& -6x_1 + 8x_2 - 3 \leq 0 \\
& 3x_1 - x_2 - 3 \leq 0 \\
& -x_i \leq 0, \; x_i - 5 \leq 0, \; i = 1,2 \\
& -x_3 + \underline \epsilon_p \leq 0, \; x_3 - 50 -  \overline \epsilon_p \leq 0 \\
& -x_4 \leq 0, \; x_4 - 10 \leq 0 \\
 & x_1 + x_2 - x_4 = 0,
\end{array}
$$

\noindent where $p(x_4) \approx x_4^2$ is a piecewise-linear approximation and where $\underline \epsilon_p, \overline \epsilon_p$ denote additional slacks added to the bound constraints on $x_3$ to account for approximation error. Both under/overapproximations are used as they lead to outer and inner RCP approximations of the feasible set, and thus allow to both lower and upper bound the cost of the original problem. In this manner, one may refine the quality of the approximation until the gap between the two becomes sufficiently small.

Noting that the global minimum of this problem lies at $(1.1667, 0.5000)$ with a cost value of $-1.0833$, the problem is solved for approximations with increasing $n_p$ values. The results, given in Table \ref{tab:ex4}, show a mild increase in computational effort as $n_p$ increases, as domain reduction techniques are able to find the solution in many cases without requiring full enumeration. Depending on the user's requirements, the procedure of increasing $n_p$ could be brought to an end once the lower and upper bounds grow sufficiently close -- for $n_p = 200$, one sees that the gap is in the fourth digit, for example, which may be sufficiently accurate. It is also worth noting that the upper bounds provided by the overapproximate RCP solution can be further tightened by a final local minimization, as the solution point here must be a feasible point for the original problem and therefore can only be improved upon by any local descent method.

\begin{table}
\begin{center}
\caption{Computational effort for Example 5. Here, $(-)/(+)$ denote under/overapproximations, respectively.}
\label{tab:ex4}       % Give a unique label
\begin{tabular}{lllllll}
\hline\noalign{\smallskip}
$n_p$ & Convex & LP & Local & Termination & $(x^*_1,\; x_2^*)$ & $c^T x^*$  \\
\noalign{\smallskip}\hline\noalign{\smallskip}
$3-$ & 0 & 23 + 1 & 1 & I & (1.4819,\;1.4457) & $-2.9326$ \\
$3+$ & 0 & 13 + 1 & 1 & I & (1.5000,\;1.5000) & $-0.2500$ \\
$5-$ & 0 & 13 + 1 & 1 & I & (1.1282,\;0.3846) & $-1.5180$ \\
$5+$ & 17 & 15 + 1 & 1 & III & (1.2350,\;0.7050) & $-1.0694$ \\
$10-$ & 0 & 17 + 1 & 1 & I & (1.1500,\;0.4500) & $-1.2431$ \\
$10+$ & 0 & 17 + 1 & 1 & I & (1.2350,\;0.7050) & $-1.0693$ \\
$15-$ & 0 & 27 + 1 & 1 & I & (1.1804,\;0.5411) & $-1.1491$ \\
$15+$ & 0 & 30 + 1 & 1 & I & (1.2350,\;0.7050) & $-1.0694$ \\
$20-$ & 0 & 42 + 1 & 1 & I & (1.1948,\;0.5843) & $-1.1170$ \\
$20+$ & 0 & 33 + 1 & 1 & I & (1.1075,\;0.3225) & $-1.0728$ \\
$30-$ & 0 & 56 + 1 & 1 & I & (1.2086,\;0.6258) & $-1.0936$ \\
$30+$ & 0 & 54 + 1 & 1 & I & (1.1500,\;0.4500) & $-1.0825$ \\
$50-$ & 3 & 117 + 2 & 2 & I & (1.1674,\;0.5021) & $-1.0888$ \\
$50+$ & 3 & 110 + 2 & 2 & I & (1.1839,\;0.5518) & $-1.0825$ \\
$100-$ & 7 & 173 + 2 & 2 & I & (1.1756,\;0.5269) & $-1.0844$ \\
$100+$ & 29 & 127 + 2 & 2 & III & (1.1585,\;0.4755) & $-1.0831$ \\
$200-$ & 203 & 208 + 3 & 5 & II & (1.1652,\;0.4956) & $-1.0835$ \\
$200+$ & 6 & 189 + 2 & 2 & II & (1.1586,\;0.4757) & $-1.0831$ \\
\noalign{\smallskip}\hline
\end{tabular}
\end{center}
\end{table}

\vspace{3mm}
\noindent {\it Example 6 (Nonlinear equality constraints and multiple global minima)}
\vspace{1mm}

Problem (\ref{eq:sinaff}) is solved by solving its RCP approximation (\ref{eq:sinaffRCP}), with the bound constraints $-3.5 \leq x_3 \leq 3.5$, $-1 \leq x_4 \leq 1$, $-3.5 \leq x_5 \leq 3.5$, and $-4.5 \leq x_6 \leq 4.5$ added to ensure the boundedness of $\mathcal{X}$. Because of the presence of the nonlinear equality constraint, the approximations used are necessarily underapproximations so as to avoid $\mathcal{X} = \varnothing$, which implies that solving the RCP problem can only provide a lower bound on the optimal cost function value for the original. For this particular problem, however, an upper bound may nevertheless be obtained by taking the solution of the RCP and using it as a starting point for a local minimization of the original problem. Like in the previous example, it is clear that finer and finer approximations may be used until the gap between the lower and upper bounds is sufficiently small -- for simplicity, the number of approximation pieces for each function is made the same, with $n_a = n_b = n_c = n_p$. Algorithm~\ref{algo:pwapp} is used to obtain the approximations $p_b$ and $p_c$, while a piecewise-linear approximation is used for $p_a$. As some numerical issues were encountered for this particular problem, the value of $\epsilon_g$ in Step 8c of Algorithm \ref{algo:RCPfull} specifically had to be raised from $10^{-6}$ to $5 \cdot 10^{-4}$ to avoid fathoming $i_{A^*}$ during the final step of the enumeration. 

Apart from the ``inconvenience'' of a nonlinear equality constraint, this problem has an additional difficulty in that the objective function exhibits a symmetry (Fig. \ref{fig:ex5}) and has two global minima at $(-1.8601,\;5.0000)$ and $(1.8601,\;-5.0000)$ with a cost value of $-3.2205$. The consequence of this is that domain reduction is unlikely to be as effective as it may be in certain problems, due to the two global minima being dispersed on nearly opposite corners of the original domain and the impossibility of shrinking the domain without fathoming one of these minima. The computational results, given in Table \ref{tab:ex5}, largely confirm these expectations, with domain reduction playing a very minor role in all cases -- this is evident from the relatively small number of LP problems solved and the fact that the algorithm never terminates by Criteria I or II. For each tested value of $n_p$, one notices that the global optima of the RCP approximation are always placed in the corners of the feasible domain, with further refinement not occurring even for $n_p = 30$. While the upper bound obtained by a local optimization following initialization at $\pm (2.0000,\;-5.0000)$ is always that of the global minimum, the refinement on the lower bound from increasing $n_p$ comes fairly slowly, suggesting that smarter, more efficient approximations are needed.

\begin{figure*}
\begin{center}
\includegraphics[width=0.75\textwidth]{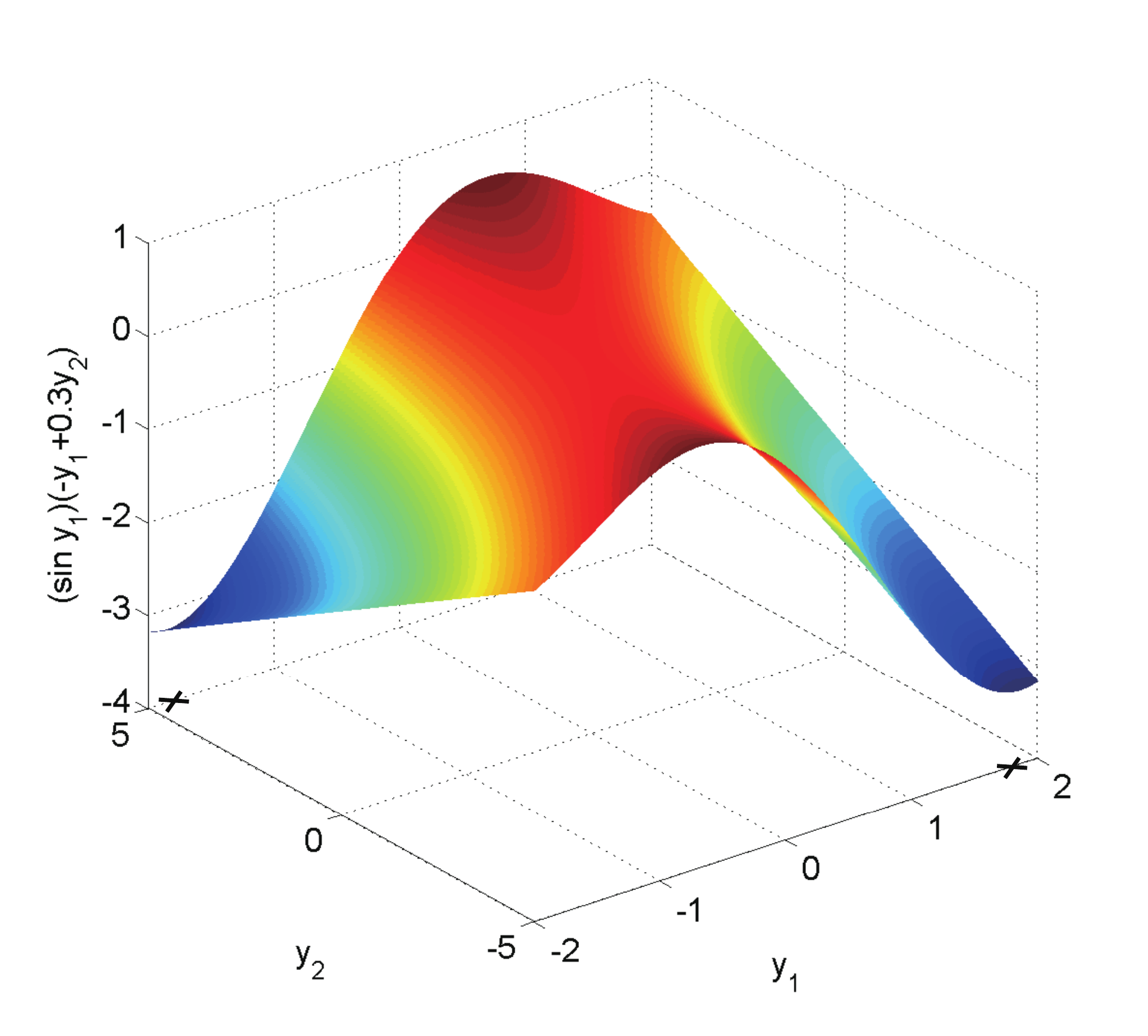}
\caption{The function $(\mathop {\sin} y_1)(-y_1 + 0.3y_2)$, which has two global minima (marked) over $y_1 \in [-2,2], \; y_2 \in [-5,5]$ at  $(-1.8601,\;5.0000)$ and $(1.8601,\;-5.0000)$.}
\label{fig:ex5}
\end{center}
\end{figure*}

\begin{table}
\begin{center}
\caption{Computational effort for Example 6.}
\label{tab:ex5}       % Give a unique label
\begin{tabular}{llllllll}
\hline\noalign{\smallskip}
$n_p$ & Convex & LP & Local & Termination & $(x^*_1,\; x^*_2)$ & $c^T x^*$ & Upper Bound \\
\noalign{\smallskip}\hline\noalign{\smallskip}
3 &  98 & 33 + 1 & 2 & III & $\pm (2.0000,\;-5.0000)$ & $-7.7443$ & $-3.2205$  \\
4 &  203 & 15 + 2 & 3 & III & $\pm (2.0000,\;-5.0000)$ & $-5.6296$ & $-3.2205$  \\
5 &  342 & 22 + 2 & 4 & III & $\pm (2.0000,\;-5.0000)$ & $-5.3914$ & $-3.2205$  \\
6 &  508 & 27 + 4 & 6 & III & $\pm (2.0000,\;-5.0000)$ & $-5.3332$ & $-3.2205$  \\
7 &  669 & 24 + 6 & 8 & III & $\pm (2.0000,\;-5.0000)$ & $-4.8414$ & $-3.2205$  \\
8 & 913 & 25 + 5 & 10 & III & $\pm (2.0000,\;-5.0000)$ & $-4.5361$ & $-3.2205$  \\
9 & 1177 & 26 + 8 & 12 & III & $\pm (2.0000,\;-5.0000)$ & $-4.4129$ & $-3.2205$  \\
10 &  1518 & 28 + 10 & 16 & III & $\pm (2.0000,\;-5.0000)$ & $-4.3582$ & $-3.2205$  \\
15 &  3361 & 31 + 16 & 32 & III & $\pm (2.0000,\;-5.0000)$ & $-3.9813$ & $-3.2205$  \\
20 &  6055 & 35 + 23 & 55 & III & $\pm (2.0000,\;-5.0000)$ & $-3.7734$ & $-3.2205$  \\
30 &  13451 & 26 + 33 & 113 & III & $\pm (2.0000,\;-5.0000)$ & $-3.5688$ & $-3.2205$ \\
\noalign{\smallskip}\hline
\end{tabular}
\end{center}
\end{table}

\vspace{2mm}
\noindent {\it General Observations}
\vspace{1mm}

That Algorithm 2 was able to find a global optimum was confirmed for all problems except the $\alpha \neq 1$ cases of Example 4, for which the global optima are neither easy to verify nor are reported in the literature. 

Particularly notable in the observed performance was the role of domain reduction, which was able to find $x^*$ in very many cases without requiring the full enumeration of active sets. This goes to further reinforce the strength and potential of these techniques, which have been an important driving force in the success of branch-and-reduce solvers \cite{Ryoo:96,Neumaier:05}. At the same time, as illustrated in Example 6, these techniques may be of little use when multiple good optima are dispersed in different corners of the feasible space. A potential solution to this could come via ``minimally partitioning'' the feasible domain and solving several RCPs in parallel -- in the case of Example 6, one could certainly envision drastic improvements in performance if the original RCP were split into two along the line $x_1 = 0$, both of which would likely be solved quickly with the help of domain reduction and then simply compared.

With regard to scaling, Example 3 represents an ideal example where the brute enumeration approach would scale linearly in $n$ and where the proposed method appears to do likewise, though the extra effort needed for domain reduction is at times sporadic. It is not yet clear how well the method would scale in general, although the results of Example 4 suggest that the required computational effort may grow significantly for $n = 10$ (and higher) if domain reduction is not effective.

Finally, Examples 5 and 6 demonstrate the viability of applying the extended RCP framework to solve factorable NLP problems to global optimality within a desired tolerance. A notable shortcoming is the apparent lack of a unified method to obtain the upper bound on the optimal cost function value in this case, as two different methods were used for the two problems presented here. For the general problem, it may be that one can neither solve the overapproximate RCP due to feasibility issues ($\mathcal{X} = \varnothing$) nor sample and perform local minimization for the original problem as the $x^*$ of the underapproximation may not yield a feasible point for the original. As such, one can only refine the underapproximation until the obtained $x^*$ becomes sufficiently close to feasible. A smarter way of managing approximations is clearly needed as well.

\section{Concluding Remarks}

The present work has put forth a methodology for solving a large class of factorable NLP problems by solving their RCP approximations. As the latter may be solved by an enumeration procedure, the proposed method offers the possibility of solving any factorable NLP satisfying Assumptions A1 and A2 to a controlled approximation error by enumeration. For certain problems, this may be advantageous since enumerative methods generally scale differently than methods iteratively partitioning the decision-variable space. Because brute enumeration of the possible solutions is usually too computationally demanding, a number of steps was taken to develop an efficient enumeration procedure that considers subsets of the active sets potentially defining the optimum. This procedure was then shown to both solve RCP problems exactly and factorable NLP problems approximately to increasing precision.

Although a basic theoretical treatment of both RCP approximation quality and the solvability of the resulting RCP problems were carried out, there are still improvements to be made with regard to both. For approximations, it may be of interest to consider a different avenue than factoring the NLP problem and then approximating its univariate components, as such approximations may require too many pieces for a sufficiently accurate approximation while also augmenting the decision-variable space. One path of potential interest is the use of multivariate D. C. functions, as it has been proven that an arbitrarily good \emph{multivariate} piecewise-concave approximation of such functions exists and may be obtained by a very cheap computational procedure \cite{Bunin:2014}.

With regard to RCP solvability, the method proposed here remains heuristic unless RCP regularity can be proven. While strict optimality of a global minimum can be fairly straightforward to prove for some problems -- such as those used in the examples -- it is not so for the general case. Furthermore, even if strict optimality were proven, one still requires the additional LICQ assumption, which cannot be proven to hold for most problems. Overcoming these challenges would provide much comfort with regard to the reliability of the method. 

On the algorithmic end, the application of the proposed method to the chosen examples represents a promising start. However, there is clearly much that needs to be done to make the method competitive, as problems where domain reduction is not successful, and/or where the number of constraints or approximation pieces is large, tend to make the method scale poorly due to its enumerative nature. As with any method employing approximations, the idea of homotopy is a natural recourse, and one can envision starting with a very brute RCP approximation with very few constraints and then iteratively refining it in the neighborhood of the $x^*$ found during each iteration. In fact, such methods could also be generalized to the case without approximations, as general concave constraints $g_i (x) \leq 0, \; i = 1,...,n_g$ could always be approximated by their joint versions $\sum_{i=1}^{n_g} g_i (x) \leq 0$, the latter leading to an easier, albeit approximate, RCP problem. The proper management of such refinements represents yet another topic for future research.

\section*{Acknowledgements}

The author would like to extend his most profound thanks to the anonymous reviewer, whose numerous suggestions have greatly contributed to the improvement of the present document.



\end{document}